\documentclass[11pt, reqno, a4paper]{amsart}

\usepackage{amsmath}
\usepackage{amsfonts}
\usepackage{amssymb}
\usepackage{bbm}
\usepackage{mathtools}
\usepackage{tikz-cd}
\usepackage{makecell}
\usepackage{booktabs}
\tikzset{
    labl/.style={rotate=270, anchor=south, inner sep=1mm}
}

\usepackage[final]{microtype}

\usepackage[
    backend=biber,
    style=alphabetic,
    natbib=true,
    url=true,
    isbn=false,
    doi=false,
    eprint=false,
    backref=true,
    ]{biblatex}
\AtEveryBibitem{%
  \ifentrytype{software}{}{%
    \clearfield{url}%
    \clearfield{urlyear}%
  }%
}
\DeclareFieldFormat{postnote}{#1}
\usepackage{amsthm}
\usepackage{float}
\usepackage[margin=1.5in]{geometry}

\usepackage[ color = blue!20]{todonotes}
\usepackage{hyperref}

\colorlet{phoen}{red!50!black}
\colorlet{purpur}{purple!45!black}
\colorlet{skog}{green!30!black}
\colorlet{hav}{blue!30!black}
\colorlet{himmel}{purpur!90!white}
\colorlet{qem}{green!45!black}
\hypersetup{
  colorlinks,
  linkcolor={red!50!black},
  citecolor={purple!45!black},
  urlcolor={green!40!black}
}

\usepackage[T1]{fontenc}

\numberwithin{equation}{section}
\theoremstyle{plain}
\newtheorem{theorem}{Theorem}[section]
\newtheorem{definition}[theorem]{Definition}
\newtheorem{proposition}[theorem]{Proposition}
\newtheorem{lemma}[theorem]{Lemma}
\newtheorem{corollary}[theorem]{Corollary}
\newtheorem{conjecture}[theorem]{Conjecture}
\newtheorem{question}[theorem]{Question}

\theoremstyle{remark}
\newtheorem{remark}[theorem]{Remark}
\newtheorem{example}[theorem]{Example}

\newcommand{\N}{\mathbb{N}}
\newcommand{\Z}{\mathbb{Z}}
\newcommand{\C}{\mathbb{C}}

\newcommand{\Q}{\mathbb{Q}}
\newcommand{\F}{\mathbb{F}}
\newcommand{\PP}{\mathbb{P}}
\renewcommand{\P}{\mathbb{P}}
\newcommand{\abs}[1]{\lvert #1 \rvert} 

\newcommand{\PSL}{\mathrm{PSL}}
\newcommand{\SL}{\mathrm{SL}}
\newcommand{\PGL}{\mathrm{PGL}}
\newcommand{\GL}{\mathrm{GL}}
\newcommand{\mf}[1]{\mathfrak{#1}}
\newcommand{\n}{\mf{n}}
\newcommand{\p}{\mf{p}}
\newcommand{\Ind}[3]{\mathrm{Ind}^{#1}_{#2}(#3)}
\newcommand{\CoInd}[3]{\mathrm{CoInd}^{#1}_{#2}(#3)}
\newcommand{\res}{\mathrm{res}}

\newcommand{\im}{\mathrm{im}}
\newcommand{\Gal}{\mathrm{Gal}}

\newcommand{\Mod}[1]{\ (\mathrm{mod}\ #1)}

\newcommand{\OO}{\mathcal{O}}

\newcommand{\sengun}{\c{S}eng\"{u}n}

\newcommand{\tors}{\mathrm{tors}}

\newcommand{\LMRefDB}[2]{\href{#1}{\texttt{#2}}}

\DeclareMathOperator{\Ad}{Ad}
\DeclareMathOperator{\Hom}{Hom}
\DeclareMathOperator{\Ext}{Ext}
\DeclareMathOperator{\Stab}{Stab}
\DeclareMathOperator{\Sh}{Sh}
\DeclareMathOperator{\tr}{Tr}
\DeclareMathOperator{\red}{red}

\DeclareMathOperator{\parab}{par}
\DeclareMathOperator{\Frob}{Frob}

\DeclareMathOperator{\norm}{Nm}

\newcommand*{\defeq}{\mathrel{\vcenter{\baselineskip0.5ex \lineskiplimit0pt
      \hbox{\scriptsize.}\hbox{\scriptsize.}}}%
  =}
\newcommand*{\eqdef}{=\mathrel{\vcenter{\baselineskip0.5ex \lineskiplimit0pt
      \hbox{\scriptsize.}\hbox{\scriptsize.}}}%
  }

\title{Lifting ethereal Bianchi modular forms}
\author{Lewis Combes \& H\aa vard Damm-Johnsen}
\date{\today}
\begin{document}

\begin{abstract}
  Mod $p$ Hecke eigenforms for congruence subgroups of $\SL_2(\Z)$ which do not lift to characteristic $0$ have received much attention, and were named {\em ethereal modular forms} by G.~Schaeffer.
  In this paper, we study liftability for mod $p$ Bianchi modular forms: computations suggest that a lift to characteristic zero may always be possible by increasing the level, as in the case of classical modular forms.
  
  We develop efficient algorithms for computing degeneracy maps in cohomology, which allow us to rigorously verify eigenvalue congruences.
  For small primes and certain fields, we also prove that lifts always exist as a consequence of the recent modularity theorem due to Caraiani--Newton.
  A \texttt{Magma} implementation, complemented by LMFDB data due to Cremona, shows that in favourable parameter ranges, most mod $p$ Bianchi modular forms lift.
\end{abstract}

\maketitle
\tableofcontents

\section{Introduction}
Let $p$ be a rational prime.
The subtle interplay between mod $p$ modular forms and characteristic $0$ modular forms has been a source of inspiration for mathematicians for several decades.
With Katz's definition of mod $p$ modular forms as sections of line bundles over modular curves over $\overline{\F}_p$, it is well-known that not every mod $p$ modular form arises from a characteristic $0$ form of the same level.
Following \cite{schaeffer2012}, we call such a form an \emph{ethereal} mod $p$ modular form.

The first example of an ethereal form in weight $1$ was found by Mestre (see \cite{edixhoven2006}) with $p =2$, and more sporadic examples were found by Buzzard \cite{buzzard2016}.
More recently, using a modularity theorem for weight $1$ forms, \cite{calegari2018} proved the existence of such forms for any prime $p \ge 3$, showing that this is in some sense a generic phenomenon.
For higher weight, \cite{kobin2025} gave a precise criterion for the existence of ethereal modular forms with level $\Gamma_0(N)$, $N \in \N$.

Now fix $d > 0$ squarefree, let $K = \Q(\sqrt{-d})$ be an imaginary quadratic field with ring of integers $\OO_K$, and fix a congruence subgroup $\Gamma_0 (\mf{n}) $ of $\Gamma_d = \PGL_2(\OO_K)$.
Following \cite{sengun2009}, we define a mod $p$ Bianchi modular form of level $\Gamma_0(\mf{n})$ and trivial weight to be a class $f \in H^{1}(\Gamma_0(\mf{n}), \F_{p})$.
This cohomology group admits a natural action of a Hecke algebra, and we call an eigenvector $f$ a mod $p$ Bianchi eigenform.
By analogy with elliptic modular forms, we say that a mod $p$ Bianchi eigenform is \emph{ethereal} if it does not arise as the reduction of a characteristic zero eigenvalue system.

Over $\Q$, one can always find a characteristic $0$ lift of an eigenform after possibly increasing the level at auxiliary primes, by the proof of Serre's conjecture due to Khare--Wintenberger.
A natural question, raised in \cite[Section 1]{fakhruddin2021} and \cite[Question 9.3.3]{calegari2012a}, is whether one may similarly lift mod $p$ Bianchi eigenforms after possibly raising the level.

\begin{question}\label{q:main-question}
  Let $f \in H^1(\Gamma_0 (\mf{n}), \overline{\F}_p)$ be an ethereal mod $p$ Bianchi eigenform.
  Does there exist an $\OO_K$-ideal $\mf{r}$ and a Hecke eigenvalue system $F$ in $ H^1(\Gamma_0 (\mf{nr}), \C)$ such that $a_\lambda (F) \equiv a_\lambda (f) \Mod p$ for all primes $(\lambda, \mf{nr} \cdot p )= 1$?\footnote{Here we are a bit imprecise about what we mean by reduction mod $p$. For a precise statement, see Definition \ref{def:liftable}.}
\end{question}
Informally, we say that such an $f$ \emph{lifts at higher level}.

For elliptic modular forms, the weight and Nebentype character play important roles in lifting questions, but we restrict our attention to trivial weight and character for simplicity, noting only that many of the computational methods extend.

By an impressive series of results (including \cite{harris2016} and, concretely, \cite[Corollary 5.4.3]{scholze2015}), a mod $p$ Bianchi eigenform $f$ gives rise to a Galois representation $\overline{\rho}_f : \Gal(\overline{K}/K) \to \GL_2 (\overline{\F}_p)$.
More precisely, for almost all primes $\lambda$ of $\OO_K$, the trace of Frobenius $\tr \overline{\rho}_f(\Frob_\lambda)$ equals the $T_\lambda$-eigenvalue $a_\lambda (f)$ of $f$.
Conversely, a putative version of Serre's conjecture for $K$, first explored by \cite{figueiredo1999}, \cite{sengun2008} and \cite{torrey2012}, predicts that certain mod $p$ Galois representations $\overline{\rho}$ are of the form $\overline{\rho}_f$ for suitable Bianchi eigenforms $f$.

As a particular case, or more precisely, by an analogue of the
Fontaine--Mazur conjecture, some Galois representations attached to mod $p$ eigenforms with $\F_p$-coefficients should arise in the $p$-torsion of elliptic curves over imaginary quadratic fields.
As suggested in \cite{calegari2012a}, when $p$ is small and the eigenvalues $a_{\lambda} (f)$ lie in $\F_{p}$, one can use this to answer Question \ref{q:main-question}.
Using the recent breakthrough due to Caraiani--Newton \cite{caraiani2023} on modularity of elliptic curves over a large proportion of imaginary quadratic fields, we have the following result.

\begin{theorem}[Proposition \ref{prop:small-prime-liftability}]
  Let $p$ be $2$, $3$, or $5$, and let $f$ be a mod $p$ Bianchi eigenform over $K$ with Hecke eigenvalues in $\F_{p}$.
  Suppose that the associated Galois representation $\rho_f$ satisfies $\det \rho_f = \chi_{\mathrm{cyc}}$, and assume that the rank of $X_{0}(15)(K)$ is $0$.
  Then $f$ lifts at higher level.
\end{theorem}

For other primes, we do not expect this strategy to extend; see Remark \ref{rmk:dieulefait}.
Instead, we give extensive computational evidence by detecting characteristic $0$ lifts of ethereal Bianchi eigenforms using data from the LMFDB \cite{lmfdb}, computed by Cremona \cite{cremona1984}.
We use algorithms based on \cite{sengun2011} to detect primes $p$ for which $H^{1}(\Gamma_0(\mf{n}), \mathbb{F}_{p})$ admits ethereal eigenforms, and using the \emph{Fox calculus}, compute $H^{1}(\Gamma_0(\mf{n}), \mathbb{F}_{p})$ and its Hecke eigenvalue systems, as well as eigenvalue systems in characteristic $0$.
If we find an apparent congruence for a fixed number of eigenvalues, we may often verify that the corresponding characteristic $0$ class lifts the mod $p$ one using the method of Section \ref{sec:verifying-lifts}.

\begin{theorem}[Theorem \ref{thm:results}]
  Across the fields $K = \Q(\sqrt{-d})$, levels $\Gamma_0(\mf{n})$ and characteristics $p$ with
  $d \in \{1,2,3,7,11\}$, $\norm(\mf{n}) \le 500$ and $2 < p < 20$, over $78.6\%$ of all ethereal mod $p$ Bianchi modular forms admit higher-level characteristic $0$ lifts.

  With no restrictions on $p$ and $\norm(\mf{n}) \le 1000$, at least $43.8\%$ lift. 
\end{theorem}

\begin{remark}
	We restrict to these particular fields because they are exactly the Euclidean imaginary quadratic fields, giving a simple word decomposition algorithm for the groups $\PSL_2(\OO_K)$. We do not pursue the same question of lifting ethereal eigenvalue systems over other fields in this paper. 
\end{remark}

For certain discriminants, we can do better: for $\Q(\sqrt{-3})$, the LMFDB data is more extensive than for the other discriminants, and we find lifts for more than $85.5\%$ of all ethereal Bianchi modular forms with coefficients in $\F_p$ and $\norm(\mf{n}) \le 1000$.
Based on these computations, we make the following optimistic prediction:

\begin{conjecture}[Conjecture \ref{conj:existence}]
  Any ethereal mod $p$ Bianchi modular form over a norm-Euclidean imaginary quadratic field $K$ admits a higher-level characteristic $0$ lift. 
\end{conjecture}

\begin{remark}
  One can reformulate Conjecture \ref{conj:existence} as a statement about global deformation rings attaining characteristic $0$ points after adding sufficient local conditions.
  In our situation, the expected dimension of the corresponding deformation space equals $0$, but since the spaces tend to be highly singular, it is difficult to obtain concrete results.
  For further details on this point of view, see \cite{calegari2009}.
\end{remark}

While liftability with $\F_p$-coefficients might not be surprising due to the connection to elliptic curves, we also find many examples of eigenvalue systems over $\overline{\F}_{p}$ which lift to characteristic $0$.

\begin{example}
  Let $\mf{n} \subset \Z[i]$ be the ideal $(2 + 17i)$ (label \texttt{293.1}; see Remark \ref{rmk:conventions}). We compute that $p=3$ is an ethereal characteristic for $\Gamma_0(\mf{n})$, and find an ethereal eigenvalue system $f$ defined over $\F_9$. 
  While $f$ does not lift to characteristic $0$ at level $\mf{n}$, we find a characteristic zero eigenvalue system $F$ at level $\mf{n r}$, with $\mf{r} = (2-i)$ (label \texttt{5.2}), congruent to $f$ mod $(3)$  for all primes $\p \nmid 3\mf{nr}$ up to norm $100$. The field generated by its Hecke eigenvalues is $E = \Q(\sqrt{5})$, in which $3$ is inert. Choosing $x^{2} + x -1$ as our minimal polynomial and letting $w$ be a root in $E$ with $\bar{w} \defeq w \Mod {3\mathcal{O}_E}$, we list the first few eigenvalues in Table \ref{table:irrational-lift-irrational-evals}.
  \begin{table}[H]
\centerline{
    \begin{tabular}{lcccccccccc}
      $\p$ & \texttt{2.1} & \texttt{5.1} & \texttt{9.1} & \texttt{13.1} & \texttt{13.2} & \texttt{17.1} & \texttt{17.2} & \texttt{29.1} & \texttt{29.2} & \texttt{37.1} \\
      \midrule
      $a_{\p}(f)$ & $2\bar{w} + 2$ & $2$ & $\bar{w}$ & $2\bar{w}$ & $2\bar{w}$ & $\bar{w} + 2$ & $2\bar{w}$ & $2\bar{w} + 1$ & $\bar{w} + 1$ & $2\bar{w} + 2$\\
      $a_{\p}(F) \Mod{3}$ & $2\bar{w} + 2$ & $2$ & $\bar{w}$ & $2\bar{w}$ & $2\bar{w}$ & $\bar{w} + 2$ & $2\bar{w}$ & $2\bar{w} + 1$ & $\bar{w} + 1$ & $2\bar{w} + 2$\\
      $a_{\p}(F)$ & $-w - 1$ & $-1$ & $w - 3$ & $2w - 3$ & $-4w - 3$ & $4w - 1$ & $2w$ & $-w - 2$ & $-5w - 5$ & $2w + 5$\\
    \end{tabular}
    }
    \caption{An ethereal mod $3$ form with Hecke eigenvalues in $\F_9$ which lifts at higher level.}\label{table:irrational-lift-irrational-evals}
  \end{table}
\end{example}

\begin{remark}\label{rmk:conventions}
	Throughout we use the LMFDB ideal labelling convention for ideals in our imaginary quadratic fields $K$, as described in \cite{cremona2020}, where \texttt{A.B} means the $B$-th ideal of norm $A$ with respect to a suitable order depending only on the minimal polynomial of $K$. In the LMFDB, the minimal polynomial chosen is either $x^2+d$ or $x^2 -x + \frac{d+1}{4}$, depending on the value of $d$ mod 4, and we follow this convention. 
\end{remark}

We also find examples of Hecke eigenvalue systems valued in $\F_p$ where the smallest level lift to characteristic $0$ is defined over an extension of $\Q$ of higher degree.

\begin{example}
  Let $K = \Q(\sqrt{-3})$, let $\omega = \frac{1 + \sqrt{-3}}{2}$, and fix $\mf{n} = (23+3\omega)$ (label \texttt{607.1}).
  We find an associated mod $31$ ethereal Bianchi form $f$ of level $\Gamma_0(\mathfrak{n})$ which has no lifts in the LMFDB.
  However, by adding $\mf{r} = (\omega-4)$ (label \texttt{13.2}) to the level, we find a characteristic $0$ lift $F$ with eigenvalues in the field $\Q(\sqrt{2})$.
  The eigenvalues of $F$ modulo the prime $\mf{q} = (4\sqrt{2}+1)$ of norm $31$ coincide with the eigenvalues of $f$, as Table \ref{table:irrational-lift-rational-evals} indicates.
  \begin{table}[H]
    \centerline{
\begin{tabular}{lcccccccccc}
  $\p$ & \texttt{3.1} & \texttt{4.1} & \texttt{7.1} & \texttt{7.2} & \texttt{13.1} & \texttt{19.1} & \texttt{19.2} & \texttt{25.1} & \texttt{31.1} & \texttt{31.2} \\
  \midrule
$a_{\p}(f)$ & $16$ & $9$ & $4$ & $13$ & $2$ & $3$ & $22$ & $22$ & $14$ & $8$ \\
$a_{\p}(F) \Mod {\mf{q}}$ & $16$ & $9$ & $4$ & $13$ & $2$ & $3$ & $22$ & $22$ & $14$ & $8$  \\
$a_{\p}(F)$ & $-2\sqrt{2}$ & $-\sqrt{2}+1$ & $4$ & $2\sqrt{2} - 2$ & $2$ & $3$ &  $\sqrt{2}-1$ & $\sqrt{2} - 1$  & $6\sqrt{2}$ & $3\sqrt{2} + 1$\\
\end{tabular}}
    \caption{An ethereal mod $31$ form with $\F_{31}$-eigenvalues lifting at higher level to a Bianchi modular form with irrational coefficients.}\label{table:irrational-lift-rational-evals}
  \end{table}
\end{example}
Several more examples are presented in Section \ref{sec:examples},
including one showing failure of multiplicity $1$ in mod $p$ cohomology when the level is highly divisible by $p$, see Example \ref{eg:mult-2}.

The cohomology groups $H^{1}(\Gamma_0(\mf{n}),\overline{\F}_{p})$ admit natural degeneracy maps to higher levels, which are made relatively explicit in Section \ref{sec:degeneracy-maps}.
To find lifts, we implement a two-stage strategy: first we find likely candidates by computing Hecke eigenvalue systems in characteristic $0$ at higher level.
Second, once we find a likely congruence, we compute the mod $p$ reduction map and level-raising maps to verify that the candidate indeed lifts the cohomology class unconditionally.

\subsection*{Outline}
In Section \ref{sec:theory}, we collect some theoretical results we need.
We begin with a precise definition of ethereal Bianchi modular forms in Section \ref{sec:ethereal-bmfs}, then prove lifting via elliptic curves in Section \ref{sec:ell-curve-lifting}.
Then in Section \ref{sec:computational-aspects}, we describe how to compute Bianchi modular forms using cohomology, which we implement in \texttt{Magma} \cite{bosma1997}. 
We also derive an explicit formula for degeneracy maps to higher level in Section \ref{sec:degeneracy-maps}, which we use to verify lifts.
Finally, in Section \ref{sec:results} we give details on the methods for, and results of, our extensive computations.

\subsection*{Acknowledgments}
We thank Christian Johansson, James Newton and John Voight for constructive discussions and many helpful suggestions.
Combes is supported by the Australian Research Council Laureate Fellowship FL230100256 grant of Geordie Williamson, and gratefully acknowledges the support of the Sydney Mathematics Research Institute.
Damm-Johnsen was supported by a postdoctoral fellowship from the Knut and Alice Wallenberg Foundation, project number KAW 2023.0427.

\subsection*{AI statement}
The code used to compute our data is virtually all human-written, with AI tools used to propose minor changes that were human-reviewed.
These tools were also used to discover potential bugs, perform quick experiments, brainstorm, and to proofread the manuscript.

\section{Theoretical aspects}\label{sec:theory}
\subsection{Ethereal Bianchi modular forms}\label{sec:ethereal-bmfs}

Let $K$ be an imaginary quadratic field of class number $1$, and define
\begin{equation}
\Gamma_{0}(\mf{n}) \defeq \bigg\{\begin{pmatrix}
  * & * \\ c & * 
\end{pmatrix}
 \in \GL_2(\OO_K) \mid  c \in \mf{n}
\bigg\},
\end{equation}
for some integral ideal $\mf{n} \subset \mathcal{O}_{K}$. 

Similarly to weight $1$ classical modular forms, see \cite[Section 8]{schaeffer2012}, Bianchi groups admit primes $p$ for which the reduction map
\begin{equation}
  \label{eq:16}
  H^{1}(\Gamma_0(\mf{n}),\Z) \otimes \F_{p} \xrightarrow{\red_{p}} H^{1}(\Gamma_0(\mf{n}), \F_{p})
\end{equation}
is not surjective. Following Schaeffer, we call such primes \emph{ethereal characteristics}. 

For a congruence subgroup $\Gamma_{1}(N) \subset \SL_{2}(\Z)$, ethereal characteristics are detected by $p$-torsion in $H^{1}(X_{1}(N), \underline{\omega})$, whereas in our setting, the cohomological degree is shifted by $1$. Indeed, it is well-known that the short exact sequence
\begin{equation}
  0 \to \Z \to \Z \to \F_{p} \to 0
\end{equation}
induces a long exact sequence which splices to
\begin{equation}
  \label{eq:17}
  0 \to H^{1}(\Gamma_0(\mf{n}),\Z) \otimes \F_{p} \xrightarrow{\red_{p}} H^{1}(\Gamma_0(\mf{n}), \F_{p}) \to H^{2}(\Gamma_0(\mf{n}),\Z)[p] \to 0.
\end{equation}

There is a natural action of $\Gamma_0(\mf{n})$ on $\P^{1}(K)$, and for $r \in \P^{1}(K)$ we let $\Gamma_0(\mf{n})_{r}$ denote the stabiliser of $r$ in $\Gamma_0(\mf{n})$.
For any $\Z[\Gamma_0(\mf{n})]$-module $M$, let
\begin{equation}
  \label{eq:18}
  H^{i}_{\parab} (\Gamma_0(\mf{n}), M) \defeq \ker \Big( H^{i}(\Gamma_0(\mf{n}),M) \to \prod_{r \in \Gamma_0(\mf{n}) \backslash \P^{1}(K)} H^{i}(\Gamma_0(\mf{n})_{r}, M)\Big)
\end{equation}
denote the group of \emph{parabolic Bianchi classes}.
When $M = \C$, these correspond to cuspidal Bianchi modular forms under the generalised Eichler--Shimura isomorphisms,
\begin{equation}
  \label{eq:20}
  H^{i}_{\parab}(\Gamma_0(\mf{n}), \C) \cong S_{2}(\Gamma_0(\mf{n})), \quad i = 1,2,
\end{equation}
see for example~\cite[Section 5.1]{ghate1999}.
Notice that the reduction map $\red_{p}$ restricts to a map
\begin{equation}
  \label{eq:21}
H^{1}_{\parab}(\Gamma_0(\mf{n}), \Z)\otimes \F_{p} \to H^{1}_{\parab}(\Gamma_0(\mf{n}), \F_{p}).
\end{equation}

\begin{definition}\label{def:liftable-class}
  We say that a class $f \in H^{1}_{\parab}(\Gamma_0(\mf{n}), \F_{p})$ is \emph{liftable} if it is in the image of $\red_{p}$.
\end{definition}

The following lemma is most likely standard, but for lack of a convenient reference, we give a quick proof.
\begin{lemma}
  Suppose $M$ is a $\Z[\Gamma_0(\mf{n})]$-module which is free of finite rank over $\Z$, and let $M^{\vee} \defeq \Hom_\Z(M, \Z)$.
  Then $H^{j}(\Gamma_0(\mf{n}), M^{\vee})_\tors \cong H_{j-1} {(\Gamma_0(\mf{n}), M)}_\tors$ for any $j \ge 1$.
\end{lemma}

\begin{proof}
  Choose a free resolution $P_\bullet \to \Z$ of $\Z[\Gamma_0(\mf{n})]$-modules, where $\Z$ has the trivial action.
  Then the homology of the complex $C_\bullet \defeq P_\bullet \otimes_{\Z[\Gamma_0(\mf{n})]} M$ is $H_\bullet (\Gamma_0(\mf{n}), M)$.
  Notice that $C_j$ is free over $\Z$ since $P_j$ is free over $\Z[\Gamma_0(\mf{n})]$ and $M$ is free over $\Z$.
  It follows that $d(C_j)$ is a free submodule of $C_{j-1}$, being a submodule of a free module over a PID.
  We may then apply the universal coefficient theorem~\cite[3.6.5]{weibel1994}, giving a split exact sequence
  \begin{equation}
    0 \to \Ext^{1}_\Z (H_{j-1}(C_\bullet), \Z) \to H^{j}(\Hom_\Z(C_\bullet, \Z)) \to \Hom_\Z (H_j(C_\bullet), \Z) \to 0.
  \end{equation}
  By the tensor-hom adjunction, we have an isomorphism of complexes
  \begin{equation}
    \label{eq:26}
    \Hom_\Z(C_{\bullet}, \Z)  \cong \Hom_{\Z[\Gamma_0(\mf{n})]}(P_\bullet, M^{\vee}),
  \end{equation}
  so the middle term computes $H^{j}(\Gamma_0(\mf{n}), M^{\vee})$.
  For the term on the left, note that $H_{j-1}(\Gamma_0(\mf{n}), M)$ is finitely generated since $\Gamma_0(\mf{n})$ is a Bianchi group, so a standard computation yields that $\Ext^{1}_\Z(H_{j-1}(\Gamma_0(\mf{n}), M), \Z) \cong H_{j-1}(\Gamma_0(\mf{n}), M)_\tors$.
  On the other hand, $\Hom_\Z (H_j(C_\bullet), \Z)$ is torsion-free.
  Since the sequence is split, the result follows.
\end{proof}
In particular, when $M = \Z \cong M^{\vee}$ and $j=2$, we infer:

\begin{corollary}
  Suppose $H_{1}(\Gamma_0(\mf{n}),\Z)[p] \ne 0$. Then $p$ is an ethereal characteristic of $\Gamma_0(\mf{n})$. 
\end{corollary}

\begin{remark}
  This translates to a simple method of finding ethereal characteristics for $\Gamma_0(\mf{n})$: given a presentation of $\PSL_{2}(\OO_{K})$, one may find a presentation of $\Gamma_{0}(\mf{n})$ through the action on $\P^{1}(\OO_{K}/ \mathfrak{n})$.
  Then the primary decomposition $H_{1}(\Gamma, \Z) \cong \Gamma/ \Gamma'$ can be computed efficiently using a computer algebra system, and from this one can read off the primes $p$ for which $H_{1}(\Gamma_{0}(\mf{n}),\Z)[p] \ne 0$.
\end{remark}

\begin{example}\label{example:large_prime_ethereal}
It is well-known that the torsion in $H_1$ of Bianchi congruence subgroups can be very large.
For example, for $K = \Q(\sqrt{-7})$, our algorithms detect that $1481759$ is an ethereal characteristic for $\Gamma_0(\mf{n})$ where $\mf{n}$ is the ideal labelled \texttt{547.2}.
\end{example}

For $R = \Z$, $\C$ or $\F_p$, the cohomology group $H^{1}(\Gamma_0(\mf{n}), R)$ admits a natural action of Hecke operators $T_{\mf{p}}$ indexed by primes $\mf{p}$ of $\OO_K$ coprime to $\mf{n}$.
In the case of $R = \F_p$, these will be described in further detail in Section \ref{sec:hecke-operators}.
An eigenvector $f$ for this action is referred to as a Bianchi eigenform, and the finite extension of $\operatorname{Frac} R$ generated by the eigenvalues is called its Hecke field, denoted $\Q(f)$.

\begin{definition}\label{def:liftable}
  Let $f$ be an ethereal mod $p$ eigenform. We say that \emph{$f$ admits a characteristic zero lift of higher level}, or \emph{is liftable at higher level} for short, if there exists an integral ideal $\mf{r} \subset \OO_K$ and a Bianchi eigenform $F$ in $H^{1}_{\parab}(\Gamma_0(\mf{nr}), \Z)\otimes \overline{\Q}$ such that for some prime $\mf{p}$ in $\Q(F)$ above $p$, we have
  \begin{equation}
    a_{\mf{q}}(F) \equiv a_{\mf{q}}(f) \Mod{\mf{p}},
  \end{equation}
  for all primes $\mf{q} \subset \OO_K$ coprime to $\mf{nr}$ and $p$.
\end{definition}
There is a subtle difference between liftable classes (Definition \ref{def:liftable-class}) and mod $p$ eigenvalue systems admitting lifts at higher levels: an eigenvalue system can lift without the corresponding cohomology class lifting.
This phenomenon occurs in Example \ref{eg:mult-2}.

\subsection{Eisenstein classes}
\label{sec:eisenstein-classes}
In characteristic $0$, the complement of $H^{1}_{\parab}(\Gamma_0(\mf{n}),\C)$ in $H^{1}(\Gamma_0(\mf{n}), \C)$ is computed by Eisenstein classes, which can be constructed explicitly.
Over $\F_{p}$, however, there is no analogous construction of non-parabolic classes.
In lieu of this, we follow \cite{calegari2012a} and define Eisenstein systems in terms of the associated Galois representation.

\begin{definition}
  Let $f$ be a mod $p$ Bianchi eigenform in $H^{1}(\Gamma_0(\mf{n}), \overline{\F}_{p})$.
  We say $f$ is \emph{Eisenstein} if there exist Hecke characters $\chi_{1},\chi_{2} : \operatorname{Cl}_{\mf{m}} \OO_{K} \to \overline{\F}_{p}^{\times}$ of some modulus $\mf{m}$ such that
  \begin{equation}
    \label{eq:22}
    a_{\mf{p}} (f) =  \chi_1(\mf{p}) + \chi_2 (\mf{p})
  \end{equation}
  for all but finitely many $\mf{p} \subset \OO_K$. Further, we say $f$ is \textit{cyclotomic-Eisenstein} if we have 
  \begin{equation*}
  	a_{\p}(f) = 1 + \norm(\p)
  \end{equation*}
  for all but finitely many $\p \subset \OO_K$. 
\end{definition}
These correspond to definitions (D0) and (D1) in \cite[Section 3.8]{calegari2012a}.
Eisenstein eigenvalue systems automatically lift to characteristic $0$;  the characters $\chi_{1}$ and $ \chi_{2}$ can be lifted to $\C$, and vectors in the automorphic representation $\operatorname{Ind}(\chi_1 \oplus \chi_2)$ correspond in some natural sense to the Galois representation $\chi_{1} \oplus \chi_{2}$, though they may not be cohomological.

\begin{remark}
  Cyclotomic-Eisenstein systems appear often in $H^1(\Gamma_0(\n),\overline{\F}_p)$, as well as being the easiest Eisenstein systems to identify. For these reasons, we exclude them from our data and analysis in Section \ref{sec:results}. 
\end{remark}

\subsection{Lifts from elliptic curves}\label{sec:ell-curve-lifting}
For small primes $p$, it is possible to use Galois theoretic methods to ensure existence of characteristic $0$ lifts.
The next proposition expands on the remark in~\cite[Section 9.3.2]{calegari2012a} following Question 9.3.3.

\begin{proposition}\label{prop:small-prime-liftability}
  Let $p \in \{2,3,5\}$, and let $f$ be a mod $p$ Bianchi eigenform over $K$ with Hecke eigenvalues in $\F_{p}$, with associated Galois representation $\rho$.
  Suppose the rank of $X_{0}(15)(K)$ is $0$.
  Then $f$ is liftable.
\end{proposition}
\begin{proof}
  The strategy is to apply the recent result of Caraiani--Newton \cite{caraiani2023}, which states that elliptic curves over $K$ where $X_{0}(15)(K)$ has rank $0$ are modular.
  If we can find an elliptic curve $E /K$ satisfying $E[p] \simeq \rho$, then by modularity there exists a characteristic $0$ Bianchi modular form $F$ such that $\overline{\rho}_{F} \simeq \rho$.
  When $p$ is $2$, this is quite explicit.
  The associated mod $2$ Galois representation $\rho$ has image contained in $\GL_2(\F_2) \simeq S_{3}$.
  Note that $\rho$ factors through a finite extension $L = \overline{K}^{\ker(\rho)}$ with $\Gal(L/K) \simeq \im(\rho)$.
  If $g \in K[x]$ is a separable cubic polynomial whose splitting field is $L$, then the elliptic curve $E_{g} : y^{2} = g(x)$ satisfies $E[2] \simeq \rho$.

  For the remaining two primes, the existence of such elliptic curves is ensured by \cite[Theorem 2]{shepherd-barron1997}. The argument boils down to noting that such elliptic curves are
  classified by $K$-points of the twisted modular curves $X_{\rho}(p)$, and for $p=3,5$, these curves are isomorphic to $\mathbb{P}^{1}_{K}$. 
\end{proof}
\begin{remark}
  Out of the norm-Euclidean imaginary quadratic fields, the condition on $K = \Q(\sqrt{-d})$ holds for $d \in \{1,2,3\}$, but not for $d \in \{7,11\}$.
  For individual elliptic curves over these fields, it is possible to check modularity using the Faltings--Serre--Livné method.
  This is recorded in the LMFDB \cite{lmfdb}.
\end{remark}
\begin{remark}\label{rmk:dieulefait}
  It would be surprising if this argument could be extended to primes $p > 5$.
  Dieulefait shows in \cite{dieulefait2004} that for any such $p$ there exists a mod $p$ Galois representation of $G_{\Q}$, satisfying similar hypotheses as in Proposition \ref{prop:small-prime-liftability}, which does not arise as the $p$-torsion of a rational elliptic curve.
\end{remark}

In fact, many $\F_p$-valued eigenvalue systems are precluded from having lifts with rational eigenvalues, as a consequence of the recently resolved Ramanujan conjecture for Bianchi modular forms.

\begin{proposition}[Ramanujan bound {\cite[Theorem A]{boxer2025}}]\label{prop:blocking}
  If $F \in H^{1}_{\parab}(\Gamma_0(\mf{nr}), \Q)$ is a characteristic $0$ Bianchi modular form of level $\Gamma_0(\mf{nr})$ with eigenvalues in $\Q$, then for any $\p$ coprime to $\mf{nr}$,  $\abs{a_{\p}(F)} \le 2 \sqrt{\norm(\p)}$. 
\end{proposition}

The following proposition shows that not only the Hecke eigenvalues, but also the reduction types
of elliptic curves corresponding to higher-level lifts of ethereal Bianchi forms are constrained.

\begin{proposition}\label{prop:ethereal-ellcurve-steinberg}
  Suppose $f$ is an ethereal mod $p$ Bianchi modular form of level $\Gamma_0(\mf{n})$, and that $F$ is a characteristic $0$ lift of level $\mf{nr}$ associated to an elliptic curve $E$ over $K$.
  If $\mf{q} \mid \mf{r}$ is a divisor such that $(\mf{q}, \mf{n}p) =1$ and $F$ is new at $\mf{q}$, then one of the following two statements is true.
  \begin{enumerate}
  \item We have $p=2$, the curve $E$ is a ramified quadratic twist of an elliptic curve with good or multiplicative reduction at $\mf{q}$, and $\mf{q}$ divides $\mf{r}$ exactly twice.     
  \item The curve $E$ has multiplicative reduction at $\mf{q}$, and $\mf{q}$ divides $\mf{r}$ exactly once. Furthermore,
  \begin{equation}
    a_{\mf{q}}(f) \equiv \pm (1 + \norm(\mf{q})) \Mod{p}.
  \end{equation}
  \end{enumerate} 
\end{proposition}
For an explicit example of case (1), see Example \ref{ex:failure-2}.

\begin{proof}
  Let $\rho_f$ and $\rho_F$ denote the Galois representations associated to $f$ and $F$, respectively, with $\rho_f \cong \overline{\rho}_F \Mod{p}$.
  We may realise $\rho_F$ as $T_p E \cong \Z_p \oplus \Z_p$, and observe that since $\rho_f$ is unramified at $\mf{q}$ but $\rho_F$ is not, the subgroup $\rho_F(I_{\mf{q}}) \subset \GL_2(\Z_p)$ is in the kernel of reduction mod $p$.
  Accordingly, $\rho_F(I_{\mf{q}})$ is a closed subgroup of a pro-$p$ group, hence itself pro-$p$. 

  By the work of Serre--Tate \cite{serre1968}, the size of the inertia image $\rho_F(I_{\mf{q}})$ is determined by the reduction type of $E$ over $K_{\mf{q}}$: it is trivial in the case of good reduction, finite in the case of potentially good reduction, and infinite in the case of potentially multiplicative reduction.
  The good reduction case is ruled out by assumption.

  In the potentially good reduction case, we know that $\rho_F(I_{\mf{q}})$ is one of the groups
  \begin{equation}
\{C_2,C_3,C_4,C_6, Q_8, C_3 \rtimes C_4, \SL_2(\F_3) \},
  \end{equation}
  by \cite{kraus1990}.
  None of these are $p$-groups when $p\ge 5$.
  For $p=3$, the only $3$-group in the list is $C_3$.
  But an elementary argument shows that $\ker(\GL_2(\Z_3) \to \GL_2(\F_3))$ contains no element of order $3$, so $\rho_F(I_{\mf{q}})$ must be trivial.
  A similarly simple computation shows that when $p=2$ there are no elements of order $4$ in $\ker(\GL_2(\Z_2) \to \GL_2(\F_2))$, so the only $2$-group possibly isomorphic to $\rho_F (I_{\mf{q}})$ is $C_2$.
  This case does occur, and by \cite[Lemma 4.2]{coppola2026}, the base change $E_{K_{\mf{q}}}$ is then a ramified quadratic twist of a curve over $K_{\mf{q}}$ with good reduction.

  It remains to rule out additive reduction in the potentially multiplicative case.
  This happens precisely when $E$ attains multiplicative reduction over a ramified quadratic extension $K'_{\mf{q}}/K_{\mf{q}}$, in which case a twist $E'$ of $E$ has multiplicative reduction over $K_{\mf{q}}$.
  Write $\rho_F \cong \rho_{E'} \otimes \psi$ with $\psi$ the corresponding quadratic character.
  We claim that $\rho_{E'}|_{G_{K_{\mf{q}}}}$ is upper-triangular.
  Indeed, we have an isomorphism of $G_{K_{\mf{q}}}$-modules $E'(\overline{K}_{\mf{q}}) \cong \overline{K}_{\mf{q}}^{\times}/q^{\Z}$ where $q \in K_{\mf{q}}$ is the Tate parameter for $E'$.
  A basis for the $p^{n}$-torsion of $E'$ is given by $\{\zeta_{p^{n}}, q^{1/p^{n}}\}$, and $\rho_{E'}$ acts like the $p$-adic cyclotomic character $\chi_{\mathrm{cyc}}$ on the first basis vector and trivially on the quotient since $q \in K_{\mf{q}}$, hence the decomposition.
  Since $\chi_{\mathrm{cyc}}$ is unramified at $\mf{q}$, we find that
  \begin{equation}
    \label{eq:15}
    \rho_F |_{I_{\mf{q}}} \cong \begin{pmatrix}
      \psi & \ast \\ 0 &\psi
    \end{pmatrix},
  \end{equation}
  so projection onto the diagonal gives a surjection $\rho_F(I_{\mf{q}}) \to \Z/2$ as $\psi$ is ramified.
  But this is impossible when $p$ is odd since $\rho_F(I_{\mf{q}})$ is pro-$p$. The claimed reduction types then follow, and the valuations of the conductors are well-known consequences.

  Finally, we prove the formula for $a_{\mf{q}}(f)$ in the case of multiplicative reduction.
  The argument above with $E$ in place of $E'$ shows that the restriction of $\overline{\rho}_F$ to the decomposition group at $\mf{q}$ has the shape
  \begin{equation}
    \overline{\rho}_F|_{G_{K_\mf{q}}} = \begin{pmatrix}
      \eta \overline{\chi}_{\mathrm{cyc}} & \ast \\ 0 & \eta
    \end{pmatrix},
  \end{equation}
  where $\eta$ is an unramified quadratic character, and $\overline{\chi}_{\mathrm{cyc}}$ is the mod $p$ cyclotomic character.
  We then compute  
  \begin{equation}
    a_{\mf{q}} (f) = \tr \rho_f (\Frob_{\mf{q}}) = \tr \overline{\rho}_F(\Frob_{\mf{q}}) =  \eta(\Frob_{\mf{q}})(\overline{\chi}_{\mathrm{cyc}}(\Frob_{\mf{q}}) + 1) = \pm (\norm(\mf{q}) + 1),
  \end{equation}
  as claimed.
\end{proof}

\begin{remark}\label{rmk:blocking}
  There exist many ethereal forms in our dataset whose Galois representations we do not expect to lift to an elliptic curve, on the basis of Proposition \ref{prop:ethereal-ellcurve-steinberg} and the Ramanujan bound.
  For example, over $K = \Q(\sqrt{-2})$, we have a mod $19$ ethereal form $f$ at level $\n = (1 - 6\omega)$ (label \texttt{73.1}), with $\omega = \sqrt{-2}$.
  Its Hecke eigenvalue at the prime $\p_{3} = (1 - \omega)$ (label \texttt{3.2}) is $5$; a lift $F$ associated to an elliptic curve (i.e. one with rational Hecke eigenvalues) must have $\abs{a_{\p_3}(F)} \leq 2\sqrt{3}$.
  However, there are no integers $n \equiv 5 \Mod{19}$ satisfying this inequality.
  This problem can be avoided if $\p_3$ appears in the level of $F$, as the Hecke eigenvalues for primes in the level of $F$ need not match those of $f$ per our conjecture.
  But we have $a_{\p_3}(f)^2 - (1+\norm(\p_3))^2 \equiv 9 \Mod{19}$, and so this cannot happen.
	
  For ethereal forms in large characteristic, there are very few possible Hecke eigenvalues for primes of small norm for which an integer lift can be chosen that also satisfies the Ramanujan bound.
  As primes satisfying the level-raising condition are also expected to be rare (see Section \ref{sec:feasible}), we typically do not expect lifts to be associated to elliptic curves for large characteristic.
\end{remark}

\subsection{Level raising following Ribet and Calegari--Venkatesh}\label{sec:level-raising}
In this section, we study the relationship between Ribet's famous level-raising result and mod $p$ forms on cohomology.
We emphasise that the ideas of this section are largely due to Calegari and Venkatesh~\cite{calegari2012a}, which we merely adapt to our setup.

Recall that there are natural degeneracy maps
\begin{equation}
  \alpha_1,\alpha_{\mf{l}} : H^{1}(\Gamma_0(\mf{n}), \F_p) \to H^{1}(\Gamma_0(\mf{nl}), \F_p)
\end{equation}
and 
\begin{equation}
  \alpha^{\vee}_1,\alpha^{\vee}_{\mf{l}} : H^{1}(\Gamma_0(\mf{nl}), \F_p) \to H^{1}(\Gamma_0(\mf{n}), \F_p)
\end{equation}
defined by restriction and corestriction of the subgroups
\begin{equation}
  \label{eq:29}
  \Gamma_0(\mf{nl}) \subset \Gamma_0 (\mf{n})
  \quad \text{and} \quad
  \begin{psmallmatrix} \lambda & 0 \\ 0 & 1 \end{psmallmatrix} \Gamma_0(\mf{nl})
  \begin{psmallmatrix} \lambda & 0 \\ 0 & 1 \end{psmallmatrix}^{-1}
  \subset \Gamma_0 (\mf{n}),
\end{equation}
for $\mf{l} = \lambda \OO_K$.

The following proposition for $\Gamma_0(N) \subset \SL_2(\Z)$, a straightforward consequence of Ribet's results and the Deligne--Serre lifting lemma, will motivate later definitions in the Bianchi setting.

\begin{proposition}\label{prop:level-raising-over-Q}
  Let $p > 3$ and $\ell$ be distinct primes, and fix $N \in \N$ coprime to both.
  Suppose $f \in H^{1}(\Gamma_0(N), \F_p)$  is a Hecke eigenform satisfying $a_\ell (f) = \pm (1+\ell) \Mod{p}$ with $\rho_f$ irreducible.
  Then there exists a number field $E$ and an eigenform $G \in H^{1}(\Gamma_0(N \ell), \OO_E)$ which is new at $\ell$, such that the reduction of $G$ modulo a prime above $p$ is in the image of the degeneracy map
  \begin{equation}
\alpha_1 \oplus \alpha_\ell : H^{1}(\Gamma_0(N), \F_p)^{2} \to H^{1}(\Gamma_0(N\ell), \F_p).
  \end{equation}
\end{proposition}

\begin{proof}
  By the Deligne--Serre lifting lemma \cite[Lemme 6.11]{deligne1974b}, the Hecke eigenvalue system associated to $f$ lifts to a character on the Hecke algebra of $H^{1}(\Gamma_0(N), \Z)$, hence gives an eigenclass $F$ in $H^{1}(\Gamma_0(N), \OO_E)$ congruent to $f$ modulo a place $\p \le \OO_E$ above $p$, where $\OO_E$ is the integer ring in some sufficiently large finite extension $E/\Q$.
  By Ribet's level-raising result~\cite[Theorem 1]{ribet1990a}, there exists an $\ell$-new form $G$ in $H^{1}(\Gamma_0(N\ell), \OO_E)$ congruent to $f$ modulo $\p$.
  The reduction mod $\p$ of $G$ has the same Hecke eigenvalues as $f$ away from $\ell$ for the Hecke algebra on $H^{1}(\Gamma_0(N\ell), \F_p)$, and so by multiplicity $1$, the reduction of $G$ lies in the span of $\alpha_1(f), \alpha_\ell(f)$.
\end{proof}

In the Bianchi setting, the Deligne--Serre lifting lemma does not apply, since it requires a characteristic $0$ lattice.
However, as in~\cite[Theorem 4.3.1]{calegari2012a}, we may use similar ideas to prove a criterion including a similar level-raising condition which can be used to construct candidate levels for which lifts of ethereal forms may exist.

\begin{definition}\label{def:level-raising}
  Let $f \in H^{1}(\Gamma_0(\mf{n}), \F_p)$ be an ethereal Bianchi modular form, and let $\mf{l} \nmid p\cdot \mf{n}$ be a prime ideal.
  We say $f$ \emph{satisfies the level-raising condition at $\mf{l}$}, or that \emph{$\mf{l}$ is a level-raising prime for $f$} if $a_{\mf{l}} (f) = \pm (1 + \norm(\mf{l})) \Mod{p}$.
\end{definition}

\begin{proposition}\label{prop:level-raising}
  Suppose $f \in H^{1} (\Gamma_0(\mf{n}), \F_p)$ is a non-Eisenstein eigenform satisfying the level-raising condition at a prime ideal $\mf{l}$, and that $p\nmid \#\OO_K^{\times}$.
  Then there exists an eigenform $g \in H^{1}_{\parab}(\Gamma_0(\mf{nl}), \F_p)$ in the kernel of the degeneracy map
  \begin{equation}
    \alpha^{\vee} \colon H^{1}_{\parab}(\Gamma_0(\mf{nl}), \F_p) \to H^{1}_{\parab}(\Gamma_0(\mf{n}), \F_p)^{2}, \quad \alpha^{\vee}(g) \defeq  (\alpha_1^{\vee}(g), \alpha^{\vee}_{\mf{l}}(g)),
  \end{equation}
  and satisfying $a_{\mf{q}}(f) = a_{\mf{q}}(g)$ for all $(\mf{q}, \mf{nl}) =1$.
\end{proposition}
The proof is essentially an adaptation of that of~\cite[Theorem 4.3.1]{calegari2012a}, which goes back to Ribet.
A crucial ingredient is Klosin's Ihara lemma for Bianchi congruence subgroups.

\begin{lemma}[{\cite[Theorem 2]{klosin2008}}]
  Suppose $p \nmid \#\OO_K^{\times}$, and let $\mf{m}$ be a non-Eisenstein maximal ideal of the Hecke algebra of $H^{1}_{\parab}(\Gamma_0(\mf{n}), \F_p)$.
  Then the localisation of the map
  \begin{equation}
    \label{eq:30}
    \alpha : H^{1}_{\parab} (\Gamma_0(\mf{n}), \F_p)^{2} \to H^{1}_{\parab}(\Gamma_0(\mf{nl}), \F_p), \quad \alpha(f_1, f_2) \defeq  \alpha_1(f_1) + \alpha_{\mf{l}}(f_2)
  \end{equation}
  to $H^{1}_{\parab}(\Gamma_0(\mf{n}), \F_p)^{2}_{\mf{m}}$ is injective.
\end{lemma}
While the theorem is originally stated for sheaf cohomology, the translation to group cohomology is easily extracted from Section 2 of the same paper.

\begin{proof}[Proof of Proposition \ref{prop:level-raising}]
  A tedious but straightforward computation shows that the matrix of $\alpha^{\vee} \circ \alpha$ is given by
  \begin{equation}
    \label{eq:31}
    \begin{pmatrix}
      1 + \norm(\mf{l}) & T_{\mf{l}} \\ 
      T_{\mf{l}} & 1 + \norm(\mf{l})
    \end{pmatrix}
  \end{equation}
  whose determinant is $0$ on the $f$-eigenspace since $\mf{l}$ is a level-raising prime for $f$.
  By Klosin's Ihara lemma, the localisation $\alpha_{\mf{m}}$ is injective, so the kernel of $\alpha^{\vee}$ is non-zero. 
\end{proof}
Note that Proposition \ref{prop:level-raising} is a statement purely in characteristic $p$, and makes no claim about existence of characteristic $0$ lifts.
However, \emph{if} the ethereal form $f$ lifts to a form $G$ of level $\Gamma_0(\mf{nl})$, then as $G$ is new at $\mf{l}$, the reduction mod $\mf{p} \mid p$, denoted $g$, will also be in the kernel of the mod $p$ degeneracy maps.
In other words, the eigenform $g$ has a chance to lift to characteristic $0$ at the same level even though the ethereal form $f$ by definition cannot.


\section{Computational aspects}\label{sec:computational-aspects}
In this section, we describe how to compute ethereal Bianchi eigenforms in $H^{1}(\Gamma_{0}(\mf{n}), \F_{p})$, and our methods for detecting lifts.
We compute the first cohomology of the Bianchi groups $\Gamma_{0}(\mf{n})$ via a method described by \c{S}eng\"{u}n \cite{sengun2011}, for which we give some details here.  
\subsection{Computing cohomology}\label{sec:computing-cohomology}
The first cohomology of any finitely generated group $G$ with coefficients in some finite rank module $V$ over a ring $R$ can be computed using the so-called \textit{Fox calculus}. In essence, one uses the $1$-cocycle relation $\phi(gh) = \phi(g)h + \phi(h)$ to express $\phi(x)$, for any $x \in G$, in terms of the values $\phi(g_i)$, for some generating set $\left\{g_i\right\}$. Relations in the group lead to relations between the values $\phi(g_i)$, which can then be expressed with linear algebra. 
\begin{example}
  When $K = \Q(\sqrt{-2})$, the principal Bianchi group is $\Gamma = \PSL_2(\Z[\sqrt{-2}])$.
  A presentation can be obtained from that of $\SL_2(\Z[\sqrt{-2}])$ in \cite[Proposition 4.3.4]{Fine1989}, and is given by
  \begin{equation*}
    \Gamma = \langle T, T_\omega, S \mid S^2 = (TS)^3 = [T,T_\omega] = [S, T_\omega]^2 = 1  \rangle, 
  \end{equation*}
  where $[g,h] = g^{-1}h^{-1}gh$ denotes the commutator.
  This isomorphism can be realised explicitly with the identifications $T = \begin{psmallmatrix}1 & 1 \\ 0 & 1\end{psmallmatrix}$, $T_{\omega} = \begin{psmallmatrix}1 & \omega \\ 0 & 1\end{psmallmatrix}$, and $S = \begin{psmallmatrix}0 & -1 \\ 1 & 0 \end{psmallmatrix}$, where $\omega = \sqrt{-2}$.
  A $1$-cocycle $\phi \in Z^1\left( \Gamma , V \right)$ must therefore satisfy the following relations:
  \begin{enumerate}
  \item $\phi(S)( 1 + S ) = 0$,
  \item $\phi(T)S( 1 + TS + (TS)^2 ) + \phi(S)(1 + TS + (TS)^2) = 0$,
  \item $\phi(T)T^{-1}(1-T_{\omega}^{-1}) + \phi(T_{\omega})(T^{-1}-1)T_{\omega}^{-1} = 0$,
  \item $\phi(T_{\omega}) (-S + S T_{\omega}^{-1} S T_{\omega} - T_{\omega}^{-1} S T_{\omega} + 1)  + \phi(S)(S + ST_{\omega}S + T_{\omega}^{-1} S T_{\omega} + T_{\omega}) = 0 $.
  \end{enumerate}

  Fixing a basis for $V$ over $R$, each of these conditions gives a condition satisfied by the values $\phi(T)$, $\phi(T_\omega)$ and $\phi(S)$ as elements of $V$. Conversely, any triple satisfying all four conditions specifies a $1$-cocycle, and so the space of cocycles is computable as a kernel of a block matrix. Similarly, a $1$-coboundary is of the form $\phi(g) = vg - v$ for some $v \in V$, and so the subspace $B^1\left(\Gamma, V \right)$ is expressible in terms of the images of $T-1$, $T_\omega-1$ and $S-1$ in $R[G]$ on $V$. The quotient of these two spaces gives $H^1\left( \Gamma , V \right)$. 
\end{example}
A $1$-cocycle can be evaluated at an element $g \in G$ using its values on the generators, the cocycle relation, and a word decomposition algorithm for $G$. For the fields we consider, the Euclidean algorithm on $\OO_K$ gives a word decomposition algorithm on $\PSL_2(\OO_K)$. 

In principle, one could compute $H^{1}(\Gamma, V)$ for $\Gamma \subset \PSL_2(\OO_K)$ in the same manner by computing a presentation for $\Gamma$.
However, it turns out to be computationally favourable to instead change the coefficient module, using Shapiro's lemma.

Namely, for groups $H \leq G$ and an $H$-module $V$, we have
\begin{equation}\label{EquationShapiroIsomorphism}
  H^1(G,\CoInd{G}{H}{V}) \simeq H^1(H,V),
\end{equation}
where 
\begin{equation*}
  \CoInd{G}{H}{V} = \{ f \colon G \to V \mid f(gh) = f(g)h, ~ \text{for all } g \in G, ~ h \in H \}
\end{equation*}
is the coinduced module.
This is a right $G$-module with the action $(f \cdot g)(x) = f(gx)$.
Let $\{h_i : i \in I \}$ be a set of right coset representatives of $H$ in $G$; an element $f$ of $\CoInd{G}{H}{V}$ is determined by its values on the $h_i$.
Consequently, when $[G : H] < \infty$, the coinduced module is isomorphic to the \emph{induced module},
\begin{equation*}
  \Ind{G}{H}{V} = \bigoplus_{i \in I} V,
\end{equation*}
where the (right) action of $g$ permutes the coset representatives. 
The ``forward'' isomorphism in Equation \ref{EquationShapiroIsomorphism} is given by restriction and evaluation at the identity element of $G$: 
\begin{equation*}
  \Sh(f) \colon h \mapsto f(h)(1_G).
\end{equation*}
In Section \ref{sec:degeneracy-maps}, we compute an explicit inverse in special cases to give an efficient formula for degeneracy maps on cohomology.

In the case of Bianchi groups, we have a convenient set of representatives for the cosets $I$:
Recall that $\GL_{2}(\OO_{K})$ acts transitively on $\P^{1}(\OO_K /\mf{n})$ by 
\begin{equation}
  \label{eq:23}
  \begin{pmatrix}
    a & b \\ c & d 
  \end{pmatrix}
  [x:y] = [a x + b y : c x + dy]. 
\end{equation}
The stabiliser of $\infty = [1:0]$ is $\Gamma_{0}(\mf{n})$, so the map $\GL_{2}(\OO_{K}) / \Gamma_{0}(\mf{n}) \to \P^{1}(\OO_{K}/ \mf{n})$ given by $\gamma \Gamma_{0}(\mf{n}) \mapsto \gamma \cdot \infty$ is a bijection.

\begin{remark}
By the Chinese remainder theorem, when $\mf{m},\mf{n} \le \OO_{K}$ are coprime, we have
\begin{equation}
  \label{eq:24}
  \PP^1(\OO_K/\mf{mn}) \cong \PP^1(\OO_K/\mf{m}) \times \PP^1(\OO_K/\mf{n}).
\end{equation}
This decomposition plays an important role for the formulas in Section \ref{sec:degeneracy-maps}.
\end{remark}

\subsection{Hecke operators}
\label{sec:hecke-operators}
Following Ash and Stevens \cite{ash1986} we define double coset operators on cohomology.
The operator associated to the matrix $\begin{bsmallmatrix} \pi & 0 \\ 0 & 1 \end{bsmallmatrix}$ for $\pi \in \OO_K$ is called the \emph{Hecke operator} at $\pi$ and is denoted by $T_{\pi}$.
By the Hecke equivariance of the Eichler--Shimura map, this double coset operator agrees with the usual Hecke operator in characteristic $0$.
The formula for the Hecke operator on cohomology can be found in many places, e.g.~in Chapter 8 of \sengun's thesis \cite{sengun2008}.
Let $f \in Z^1(\Gamma_0(\n),V)$, and fix a set of coset representatives $\{ r_i \colon i \in I \}$ of $\Gamma_0(\n\p)$ in $\Gamma_0(\n)$.
Then, fixing a generator $\p = (\pi)$, we have
\begin{equation}\label{eq:hecke-op}
  (T_{\pi} f)(X) = \sum_{i \in I} f(g r_i X r_{X(i)}^{-1} g^{-1} )g r_{X(i)},
\end{equation}
where $X(i)$ is the index of the coset $\Gamma_0(\n\p)r_i X$, and $g = \begin{bsmallmatrix} \pi & 0 \\ 0 & 1 \end{bsmallmatrix}$. When $(\n,\p) = 1$, the decomposition given in Equation \ref{eq:24} lets one instead use the coset representatives of $\Gamma_0(\p)$ in $\PSL_2(\OO_K)$. 

From the perspective of Galois representations, we are primarily interested in automorphic forms for $\PGL_2(\OO_K)$, rather than $\PSL_2(\OO_K)$.
This is because the Hecke operators in the $\PGL_2$ space satisfy $T_{\pi} = T_{u \pi}$ for all $u \in \OO_K^{\times}$.
That is, Hecke operators are associated to ideals rather than elements, and the associated mod $p$ Galois representation should have image in $\PGL_2(\F_{q})$.
In this case, we write the Hecke operator as $T_{\p}$.
The Hecke operators are compatible with the maps from Shapiro's lemma by \cite[Lemma 1.1.4]{ash1986}.

The subspace of $\PGL_2$ classes can be obtained from the $\PSL_2$ classes by introducing the \emph{star operator}.
Let $u$ be a generator of the unit group $\OO_K^{\times}$; for all imaginary quadratic fields except $\Q(\sqrt{-1})$ and $\Q(\sqrt{-3})$, we have $u = -1$.

\begin{proposition}
  Suppose $p > 2$. Mod $p$ cohomology classes for the larger group $\PGL_2(\OO_K)$ are precisely those classes $f$ for which the star operator
\begin{equation}
  f \mapsto f^{*}(X) \defeq f(JXJ^{-1})J
\end{equation}
has eigenvalue $1$, with $J = \begin{psmallmatrix}
  u & 0 \\ 0 & 1
\end{psmallmatrix}$.
\end{proposition}

\begin{proof}
  We have a short exact sequence
  \begin{equation}
    0 \to \PSL_2(\OO_K) \to \PGL_2(\OO_K) \xrightarrow{\det} \OO_K^{\times}/ (\OO_K^{\times})^{2} \cong \Z/2\Z \to 0,
  \end{equation}
  where the cokernel is generated by $u$. 
  The corresponding inflation-restriction sequence reads
  \begin{equation}
    \label{eq:25}
    0 \to H^{1}(\Z/2\Z, \F_p) \to H^{1}(\PGL_2(\OO_K),\F_p) \to H^{1}(\PSL_2(\OO_K),\F_p)^{*} \to H^{2}(\Z/2\Z, \F_p).
  \end{equation}
  When $p > 2$, the first and last terms vanish, proving the claim.
\end{proof}

One may think of the star operator as a ``Hecke operator at infinity''.

As we represent cocycles as tuples indexed by the generators of the group, it is more efficient to compute the $\PSL_2$ cohomology and cut down using the star operator, than it is to compute the $\PGL_2$ cohomology directly. 





\subsection{Computing degeneracy maps}
\label{sec:degeneracy-maps}

Given coprime ideals $\n$ and $\mf{r}$, there are inclusions $\alpha_{i} \colon \Gamma_0(\mf{nr}) \to \Gamma_0(\n)$ indexed by the divisors $(d_i)$ of $\mf{r}$.
Explicitly, these are given by
\begin{equation*}
  \alpha_i(\gamma) = \begin{psmallmatrix}
    d_i & 0 \\ 0 & 1
  \end{psmallmatrix} \gamma \begin{psmallmatrix}
    d_i & 0 \\ 0 & 1
  \end{psmallmatrix}^{-1}. 
\end{equation*}
These induce maps in cohomology
\begin{equation*}
  \alpha_i \colon H^1(\Gamma_0(\n), V) \to H^1(\Gamma_0(\mf{nr}), V^{\alpha_i}),
\end{equation*}
called \emph{degeneracy maps}.
Here $V^{\alpha_i}$ denotes the $\Gamma_0(\mf{nr})$-module $V$ with the twisted action
\begin{equation}
  \label{eq:28}
  \gamma \cdot v = \alpha_i (\gamma) \cdot v.
\end{equation}


The degeneracy maps can be computed at the level of cohomology of $\Gamma_d$ via the inverse Shapiro map, as in the following diagram.
\begin{equation}
  \label{eq:shapiro-diagram}
\begin{tikzcd}
    {H^1(\Gamma_d,\CoInd{\Gamma_d}{\Gamma_0(\n)}{V})} & {H^1(\Gamma_d,\CoInd{\Gamma_d}{\Gamma_0(\mf{nr})}{V})} \\
    {H^1(\Gamma_0(\mf{r}),\CoInd{\Gamma_d}{\Gamma_0(\n)}{V})} & {H^1(\Gamma_d,\CoInd{\Gamma_d}{\Gamma_0(\mf{r})}{ \CoInd{\Gamma_d}{\Gamma_0(\n)}{V} })}
    \arrow["{\alpha_i^*}", dashed, from=1-1, to=1-2]
    \arrow["\res"', from=1-1, to=2-1]
    \arrow["{\Sh^{-1}}"', from=2-1, to=2-2]
    \arrow["{\sim}"' labl, from=2-2, to=1-2]
  \end{tikzcd}
\end{equation}

\begin{remark}
  In principle, one can compute the degeneracy maps by passing to full level $\Gamma_d$ and using the word decomposition.
  For our purposes, however, this is too computationally expensive.
  Instead, we will give a formula based on an explicit inverse of the Shapiro morphism.
  This lets us compute degeneracy maps with $[\Gamma_0(\mf{n}): \Gamma_0(\mf{nr})]$ word decompositions instead of $[\Gamma_d:\Gamma_0(\mf{nr})]$, giving large computational savings when $\mf{r}$ has large norm.
  A different version of an explicit Shapiro inverse, which is not well-adapted to our setup, may be found in \cite{zavarnitsine2025}.
\end{remark}

To give an explicit formula for the level-raising maps $\alpha_i$, we abstract the setup slightly.
Fix finite index subgroups $H' \subset H \subset G$, along with sets of coset representatives $T$ and $R$ satisfying
\begin{equation}
  \label{eq:2}
G = \bigsqcup_{t \in T} H t \quad \text{and} \quad H = \bigsqcup_{r \in R} H' r;
\end{equation}
in particular,
\begin{equation}
  \label{eq:5}
G = \bigsqcup_{(r,t) \in R \times T} H'r t.
\end{equation}

We make the important assumption that the action of $H$ on $R$ extends to $G$.

\begin{lemma}\label{lemma:K-group}
  A right $G$-action on $R = H'\backslash H$ extending the natural $H$-action is equivalent to the choice of a subgroup $K \subset G$ such that $H \cap K = H'$, $K H = G$ and $R \cong K \backslash G$.
\end{lemma}

\begin{proof}
  Suppose $G$ acts on $R$, and define $K \defeq \Stab_{G}(H'e)$. Transitivity immediately implies that $R \cong K \backslash G$, and evidently $H \cap K = H'$. For any $g \in G$, write $H'g = H' h$ for some $h \in H$, so $h g^{-1} \in K$. Choosing $k$ such that $h g^{-1} = k^{-1}$, we get $g = k h$, so $G = KH$.

  For the converse, simply transfer the right action from $K \backslash G$ to $R$. 
\end{proof}

Let $(g,r) \mapsto \sigma_{g}(r)$ be the {\em inverse} of the action map, so that $\sigma_{g}(r) g r^{-1} \in \Stab_{G}(H' e)$ for all $r \in R$.
In particular, this implies that $\sigma_{g_{1}g_{2}}(r) = \sigma_{g_{1}}(\sigma_{g_{2}}(r))$.
We denote by $k[H]$ the group ring of $H$ over a field $k$, and let $V$ be a right $k[H]$-module.
Suppose $\tilde{G}$ is a group containing $G$, and that $\alpha : H' \to H$ is given by conjugation by an element $\gamma$ in $\tilde{G}$ such that $\gamma H' \gamma^{-1} \subset H$ and $\gamma K \gamma^{-1} \subset G$.
We furthermore assume that the action of $H$ on $V$ extends to an action of $\langle G,\gamma \rangle$.

By Lemma \ref{lemma:K-group}, there is a $G$-equivariant bijection 
\begin{equation}
  H' \backslash G \to (K \backslash G) \times (H \backslash G),  \quad  H'g \mapsto (Kg, Hg),
\end{equation}
giving an identification $\CoInd{G}{H'}{V} \cong k[R] \otimes M$, where $M\defeq \CoInd{G}{H}{V}$.
This isomorphism respects the $G$-action if we set 
\begin{equation}\label{eq:weird-action}
  (r \otimes m) \cdot g = \sigma_g^{-1}(r) \otimes (m \cdot g),
\end{equation}
for $r \in R$, $m \in M$, and $g \in G$.
Furthermore, we define an action of $\gamma$ on $M$ by
\begin{equation}
  (Ht \otimes v) \cdot \gamma \defeq (Ht \cdot \gamma) \otimes (v \cdot \gamma).
\end{equation}

Given a cocycle $f \in Z^{1}(G, M)$, we may now define $\Phi(f) : G \to k[R] \otimes M$ by 
\begin{equation}
  \label{eq:7}
\Phi(f) (g) = \sum_{r \in R } r \otimes \left(f(\Ad(\gamma)(\sigma_{g}(r) g r^{-1})) \cdot \gamma r\right),
\end{equation}
where $\Ad(\gamma) (g) \defeq \gamma g \gamma^{-1}$.

\begin{proposition}
  Fix a cocycle $f \in Z^{1}(G,\CoInd{G}{H}{V})$.
  \begin{enumerate}
  \item $\Phi(f) \in Z^{1}(G, \CoInd{G}{H'}{V})$.
  \item If $f$ is a coboundary, then so is $\Phi(f)$. 
  \end{enumerate}
\end{proposition}

\begin{proof}
  (1) We need to verify the cocycle condition for $\Phi(f)$, namely $\Phi(f) (g_{1}g_{2}) = \Phi(f)(g_{1})\cdot  g_{2} + \Phi(f)(g_{2})$ for $g_{1},g_{2} \in G$.
  By Equation \ref{eq:weird-action}, notice that
  \begin{align}
  \label{eq:10}
    \Phi(f)(g_{1}) \cdot g_{2} &= \sum_{r} \sigma_{g_{2}}^{-1}(r) \otimes \left(f(\gamma\sigma_{g_{1}}(r) g_{1} r^{-1} \gamma^{-1}) \cdot \gamma r g_{2}\right) \\
  &= \sum_{r} r\otimes \left(f(\gamma \sigma_{g_{1}}(\sigma_{g_{2}}(r)) g_{1} \sigma_{g_{2}}(r)^{-1} \gamma^{-1}) \cdot \gamma \sigma_{g_{2}}(r) g_{2}\right),
\end{align}
by reindexing the sum under the map $r \mapsto \sigma_{g_{2}}(r)$.
Next we manipulate the inner expressions; write
\begin{align}
  \label{eq:11}
  \gamma \sigma_{g_{1}}(\sigma_{g_{2}}(r)) &g_{1} \sigma_{g_{2}}(r)^{-1}\gamma^{-1} = \\
   &\underbrace{\left(\gamma \sigma_{g_{1}g_{2}}(r) g_{1}g_{2}r^{-1}\gamma^{-1}\right)}_{\eqdef A} \underbrace{\left(  \gamma r g_{2}^{-1} \sigma_{g_{2}}(r)^{-1}\gamma^{-1}\right)}_{\eqdef B}.
\end{align}
By the cocycle identity and the fact that $f$ satisfies $f(B) = - f(B^{-1}) \cdot B$,
we get
\begin{align}
  f(AB) &= f(A) \cdot B + f(B) \\
  &= (f(A) - f(B^{-1})) \cdot B.
\end{align}
Now note that
\begin{equation}
  \label{eq:8}
B \cdot \gamma \sigma_{g_{2}}(r) g_{2} = \gamma r,
\end{equation}
and that $B^{-1} = \gamma \sigma_{g_{2}}(r)g_{2} r^{-1} \gamma^{-1}$, so
\begin{align}
  \Phi(f) (g_1) \cdot g_2
  =& \sum_{r} r \otimes (f(A) - f(B^{-1})) \cdot \gamma r \\
  =& \sum_{r} r \otimes f(\gamma \sigma_{g_{1}g_{2}}(r) g_{1}g_{2}r^{-1}\gamma^{-1}) \cdot \gamma r \\
  &-
    \sum_{r} r \otimes f(\gamma \sigma_{g_{2}}(r)g_{2} r^{-1} \gamma^{-1}) \cdot \gamma r \\
  =& \Phi(f) (g_1g_2) - \Phi(f) (g_2).
\end{align}
Rearranging, we find
\begin{equation}
  \label{eq:12}
\Phi(f)(g_{1}g_{2}) = \Phi(f)(g_{1}) \cdot g_{2} + \Phi(f)(g_{2}),
\end{equation}
as required.

(2) If $f(g) = m \cdot g - m$ for some $m \in \CoInd{G}{H}{V}$, define $w = \sum_{r} r \otimes (m \cdot \gamma r)$.
Then, unfolding definitions and applying Equation \ref{eq:weird-action}, we find

\begin{align}
  \label{eq:13}
  \Phi(f)(g)
  &= \sum_{r} r \otimes \left(f(\gamma \sigma_{g}(r) g r^{-1} \gamma^{-1}) \cdot \gamma r\right) \\
  &= \sum_{r} r \otimes \left(m \cdot \gamma \sigma_{g}(r) g r^{-1} \gamma^{-1} \cdot \gamma r - m \cdot \gamma r\right) \\
  &= \sum_{r} r \otimes \left(m \cdot \gamma \sigma_{g}(r) g - m \cdot \gamma r\right) \\
  &= \Big(\sum_{r}  \sigma_{g}^{-1}(r) \otimes (m \cdot \gamma r g )\Big) - \sum_{r} r \otimes (m \cdot \gamma r)\\
  &= w \cdot g - w,
\end{align}
which proves the second claim.
\end{proof}

Consequently, $\Phi$ induces a map on cohomology.

\begin{theorem}
  The following diagram commutes.

  \begin{equation}
    \label{eq:9}
\begin{tikzcd}
{H^1(H,V) } \arrow[r, "\alpha^\ast"]                                & {H^1(H',V^{\alpha})}                                        \\
{H^1(G,\CoInd{G}{H}{V})} \arrow[r, "\Phi"] \arrow[u, "\mathrm{Sh}_{H}"] & {H^1(G,\CoInd{G}{H'}{V})} \arrow[u, "\mathrm{Sh}_{H'}"]
\end{tikzcd}
  \end{equation}
\end{theorem}

\begin{proof}
  Recall that the map $\Sh_{H}$ is induced by the map $\sum_{t} t \otimes f_{t}(g) \mapsto f_{1}(h)$.
  We then compute
  \begin{equation}
    \label{eq:14}
\alpha^{*} (\Sh_{H}(f)) (h') = f_{1}(\gamma h' \gamma^{-1}) = 
\Sh_{H'} \Phi(f) (h'),
\end{equation}
as required.
\end{proof}

\begin{remark}
  To give a formula for the maps $\alpha_i^{*}$ from Equation \ref{eq:shapiro-diagram}, take $H = \Gamma_{0}(\mf{n})$,
  $H' = \Gamma_{0}(\mf{n r})$, $K = \Gamma_{0}(\mf{r})$, and $V =\F_p$.
  The required extension of the $G$-action to $H \backslash G \cong \PP^{1}(\OO_K/\mf{n})$ comes from reducing modulo $\mf{n}$ and right-multiplying; the action of the matrices $\gamma =
   \begin{psmallmatrix}
    d_i & 0 \\ 0 & 1
  \end{psmallmatrix}
  $
  is defined similarly.
  
  Using the decomposition from Equation \ref{eq:24}, we can use the components of an explicit cocycle $f$ of level $\Gamma_0(\mf{n})$ to quickly compute $\alpha_i^{*}(f)$ as an element of $H^{1}(\PSL_2(\OO_K), \CoInd{\Gamma_d}{\Gamma_0(\mf{nr})}{\F_p})$. 
\end{remark}

\subsection{Verifying lifts}
\label{sec:verifying-lifts}
Given a mod $p$ Bianchi eigenform, one can search for lifts by comparing some finite number of eigenvalues to others appearing in characteristic $0$. However, matching at finitely many eigenvalues does not guarantee that \textit{all} eigenvalues will match. We describe a method to verify that a probable lift actually is a lift, using degeneracy maps. 

Let $f \in H^1(\Gamma_0(\n),\F_p)$, and suppose we have $F \in H^1(\Gamma_0(\mf{nr}),\Z)$ whose eigenvalues match those of $f$ modulo $p$, up to some finite bound. For each divisor $d$ of $\mf{r}$, we have a degeneracy map $\alpha_d \colon H^1(\Gamma_0(\n),\F_p) \to H^1(\Gamma_0(\mf{nr}),\F_p)$, and each map is Hecke equivariant outside $\mf{nr}$. Alongside this, we have the ``reduction mod $p$'' map $\overline{{}\cdot {}} \colon H^1(\Gamma_0(\mf{nr}),\Z) \to H^1(\Gamma_0(\mf{nr}),\F_p)$, also Hecke equivariant. If 
\begin{equation*}
	\overline{F} \in \mathrm{span}_{\F_p}\left\{ \alpha_d(f) \colon  (d) \mid \mf{r} \right\},
\end{equation*}
then the Hecke equivariance immediately implies that 
\begin{equation*}
	a_{\mf{q}}(F) \equiv a_{\mf{q}}(f) \Mod{p}
\end{equation*}
for all $\mf{q} \nmid \mf{n r}p$.
When $f$ or $F$ have irrational eigenvalues, the same arguments hold after suitably replacing $\Z$, $p$ and $\F_p$ with $\OO$ the ring of integers of a number field, $\p$ a prime ideal in $\OO$ over the characteristic, and $\F_{\p} = \OO/\p$.
For the rest of the paper we will write equalities such as $\overline{F} = \alpha_1(f) - \alpha_{d}(f)$, which are only literally true up to a constant multiple, depending on the scaling of the cocycles involved.
This constant does not affect the congruence implied by the equality, and so we ignore it.

When $f$ is ethereal, it is associated to a torsion class $\tau_f \in H^2(\Gamma_0(\n),\Z)[p]$, and the existence of such a lift corresponds to finding some linear combination $\beta = \sum x_d \alpha_d$ of degeneracy maps such that $\beta(\tau_f) = 0$. This can be seen by the commutativity, and exactness of the rows, of the diagram 
\[\begin{tikzcd}[column sep=0.7cm]
	0 & {H^1(\Gamma_0(\n),\Z)\otimes \F_p} & {H^1(\Gamma_0(\n),\F_p)} & {H^2(\Gamma_0(\n),\Z)[p]} & 0 \\
	0 & {H^1(\Gamma_0(\mf{nr}),\Z)\otimes \F_p} & {H^1(\Gamma_0(\mf{nr}),\F_p)} & {H^2(\Gamma_0(\mf{nr}),\Z)[p]} & 0
	\arrow[from=1-1, to=1-2 ]
	\arrow["{\mathrm{red}_p}", from=1-2, to=1-3]
	\arrow["\beta", from=1-2, to=2-2]
	\arrow[from=1-3, to=1-4]
	\arrow["\beta", from=1-3, to=2-3]
	\arrow[from=1-4, to=1-5]
	\arrow["\beta", from=1-4, to=2-4]
	\arrow[from=2-1, to=2-2 ]
	\arrow["{\mathrm{red}_p}", from=2-2, to=2-3]
	\arrow[from=2-3, to=2-4]
	\arrow[from=2-4, to=2-5]
\end{tikzcd}\]
formed from two copies of Equation \ref{eq:17}. 

Not all probable lifts can be verified in this way, as it requires the mod $p$ reduction of the lifting form $F$ to be in the image of the degeneracy maps from level $\Gamma_0(\mf{n})$.
This situation is complementary to Proposition \ref{prop:level-raising} where the existence of a mod $p$ ``newform'' is proved, but no guarantee of lifting is made.

\begin{example}
	Over $K=\Q(\sqrt{-1})$, at level $\n = (4+9i)$ (label \texttt{97.1}) there is an ethereal form $f$ in the mod $5$ cohomology. Adding $\p_{2} = (1+i)$ (label \texttt{2.1}) to the level, we find a lift to characteristic $0$ at level $\n \p_{2}$, associated to the elliptic curve 
	\begin{equation*}
		E \colon y^2 + xy + iy = x^3 + (i-1)x^2 - (4 + 4i)x - 4i
	\end{equation*}
	with LMFDB label \LMRefDB{https://www.lmfdb.org/EllipticCurve/2.0.4.1/194.1/a/1}{2.0.4.1-194.1-a1}. The $\p_2$ eigenvalue of $f$ is $2$, with $a_{\p_2}^2 - (1+\norm(\p_2))^2 \equiv 0 \Mod{5}$, and so $f$ satisfies the level-raising condition at $\p_2$. 
	
	Using LMFDB data, we find further probable lifts of $f$ to characteristic $0$.
        Let $\p_{5} = (1-2i)$, $\p_{13} = (3+2i)$, $\p_{17} = (1+4i)$, and $\p_{1021} = (11+30i)$ (labels \texttt{5.1}, \texttt{13.2}, \texttt{17.2} and \texttt{1021.2}).
        Then $f$ has probable lifts to characteristic $0$ at levels $\n \mf{d}$, with $\mf{d} \in \{ \p_{5}^2, \p_{2}\p_{5}^2, \p_{2}\p_{13}\p_{17}, \p_{1021} \}$.
        Computing Hecke eigenvalues, we can verify that the conclusion of Proposition \ref{prop:ethereal-ellcurve-steinberg} holds in each case:
	\begin{align*}
		a_{\p_{13}}^2 - (1+\norm(\p_{13}))^2 \equiv 1^2 - 14^2 \equiv 0 \Mod{5}, \\ 
		a_{\p_{17}}^2 - (1+\norm(\p_{17}))^2 \equiv 3^2 - 18^2 \equiv 0 \Mod{5}, \\
		a_{\p_{1021}}^2 - (1+\norm(\p_{1021}))^2 \equiv 2^2 - 1022^2 \equiv 0 \Mod{5}, 
	\end{align*}
	with $\p_{5}$ dividing the characteristic. 
\end{example}
\begin{example}
  We give an example of a probable lift that is not in the image of the degeneracy maps, and so cannot be verified with this method.
  Over $K = \Q(\sqrt{-3})$, at level $(17)$ (label \texttt{289.1}), there is a mod $2$ ethereal form $f$.
  We find a probable lift to characteristic $0$ at level $(68)$ (label \texttt{4624.1}), after adding two copies of the prime $(2)$ to the level.
  The lift is one of a pair of Galois conjugate forms over $\Q(\sqrt{5})$.
  However, the image of the characteristic $0$ cohomology space in the mod $2$ space does not intersect the span of the degeneracy maps applied to $f$.
  That is, while the eigenvalues appear to match (up to norm $100$ at least), we cannot prove the lifting.
\end{example}

For a complete discussion of which probable lifts our computations validated at the level of cocycles, see Section \ref{sec:results-of-comps}.

\subsection{Remarks on feasibility of lifting}\label{sec:feasible}
Lifts of an ethereal form $f$ typically occur after adding primes to the level at which $f$ satisfies the level-raising condition, or those that divide the level or characteristic.
In each case, large primes present a difficulty, as the level of a potential lift ends up being larger, and more expensive to compute.

Large characteristics also pose another problem: primes for which $f$ satisfies the level-raising condition appear to become increasingly rare.
For example, over $\Q(i)$, the mod $37$ ethereal form at level $(12-13i)$ (label \texttt{313.2}) satisfies the level-raising condition at $17$ primes of norm at most $1000$.
Meanwhile, the ethereal form in the mod $137$ cohomology at level $(3+20i)$ (label \texttt{409.2}) has only two such primes with norm at most $1000$.

Heuristically, the residues of $1+\norm(\mf{q})$ should be uniformly distributed mod $p$, as should the eigenvalues $a_{\mf{q}}(f)$ of an ethereal form $f$, at least if $\rho_f$ is assumed to have full image, by Chebotarev.
Then $f$ should satisfy the level-raising condition at a prime $\mf{q}$ with probability about $\frac{2}{p}$.
Even when $f$ \textit{does} satisfy the level-raising condition at a prime, there is no guarantee that adding it to the level once will be enough to give a lift.
For eigenvalue systems like those associated to the mod $1481759$ ethereal form of Example \ref{example:large_prime_ethereal}, the computational feasibility of finding lifts to characteristic $0$ in this way seems vanishingly small.

\section{Results \& implementation details}
\label{sec:results}

In this section we describe how to implement the aforementioned algorithms and the resulting data computed.
The code may be found in \cite{combes2026}.

\subsection{Using LMFDB eigenvalue data}
\label{sec:using-lmfdb}

At the time of writing, the LMFDB stores data describing all rational newforms in $H^1(\Gamma_0(\n),\C)$ for $\mf{n}$ as tabulated in Table \ref{table:lmfdb-data}, including hundreds of eigenvalues for each eigenform.

\begin{table}[H]
  \begin{tabular}{cccc}
    $d$& $2, 7, 11$ & $1$ & $3$  \\ \midrule
    $\norm(\mf{n})$& $50{,}000$ & $100{,}000$ & $150{,}000$   \\
  \end{tabular}
  \caption{Upper bounds of levels for Bianchi eigenforms for $\Q(\sqrt{-d})$ in the LMFDB.}\label{table:lmfdb-data}
\end{table}

This data can be used to quickly search for lifts.
We compute the eigenvalues of some form $f \in H^1(\Gamma_0(\n),\F_{\p})$ up to some bound, say $\norm(\mf{q}) \leq 100$, which can be compared to those in the database appearing at levels $\mf{nr}$.
Probable lifts are those with eigenvalues matching modulo $p$ outside the support of $\mf{nr}p$, which can then often be checked explicitly by the method outlined in Section \ref{sec:verifying-lifts}. For more details on explicitly verified lifts in our data, see Section \ref{sec:results-of-comps}.

In practice, we compute the mod $p$ eigenvalue systems using the Fox calculus implemented in the \texttt{Magma} computer algebra system \cite{bosma1997}, and store the results in csv files.
These are then loaded into \texttt{SageMath} \cite{thesagedevelopers2022}, and compared with LMFDB \cite{lmfdb} eigenvalue systems fetched using David Roe's \href{https://github.com/roed314/lmfdb-lite}{\texttt{lmfdb-lite} library}.
Then we make another pass in \texttt{Magma}, looking for lifts with irrational Hecke eigenvalues using builtin functionality.
Finally, we call the code to verify the lifts in \texttt{Magma}, accessed via \texttt{SageMath}. Details of the lifts found this way can be found in Section \ref{sec:results-of-comps}.




Based on our extensive computations, we make the following (optimistic) predictions:

\begin{conjecture}\label{conj:existence}
  Let $p$ be a rational prime, and let $f$ be an ethereal mod $p$ Bianchi eigenform with trivial coefficients over a norm-Euclidean imaginary quadratic field $K$ of level $\mf{n}$.
  Then the form $f$ admits a higher-level characteristic $0$ lift.
\end{conjecture}
It would not be surprising if the conjecture, if true, holds for other fields, but these are beyond the scope of our current computations.

\begin{remark}
  For $\GL_2(\Q)$, \cite[Theorem 3]{khare1995} implies that adding $p^{2}$ to the level is always sufficient to find a lift.
    This does not appear to hold for Bianchi modular forms; for example, for $K = \Q(\sqrt{-1})$ there is a mod $3$ eigenvalue system of prime level \texttt{277.1} which does not lift to characteristic $0$ after adding the inert prime $3$ or $3^{2}$ to the level.
\end{remark}
We summarise our computational findings as follows:

\begin{theorem}\label{thm:results}
  Conjecture \ref{conj:existence} holds true for at least $315$ out of $401$
  mod $p$ ethereal forms of level $\Gamma_0 (\mf{n})$ with $\mf{n}$ of norm $\le 500$ and $2 < p < 20$, and for at least $1144$ out of $2608$ ethereal forms with $\mf{n}$ of norm $\leq 1000$ and no restriction on $p$.
\end{theorem}

When the ethereal mod $p$ eigensystem $f$ has no lift in the LMFDB, it is still possible that lifts exist, but have Hecke eigenvalues defined over a number field $E$ of degree $d>1$.
It is no longer feasible to check the characteristic $0$ systems of every level $\mf{nr}$ up to a bound.
Instead, we check only the ideals $\mf{r}$ where the divisors $\mf{q}$ satisfy the level-raising condition in Definition \ref{def:level-raising}.

We use efficient algorithms, using the Voronoï polyhedron and implemented by Dan Yasaki in the \texttt{Magma} computer algebra system, for computing characteristic $0$ Bianchi modular forms.
Another alternative is the specialised modular symbols code written by Cremona and his collaborators, which is the origin of the LMFDB data.

We use these as a quick first pass to look for Hecke eigenvalue systems valued in $E$ which appear to lift $f$, and then use the methods of Section \ref{sec:verifying-lifts} to verify.

\begin{remark}\label{rmk:hard-lifts}
  While we were able to lift only around $4$ in every $10$ ethereal forms across all fields up to norm $1000$, we nevertheless conjecture that lifts can always be found in general.
  The dataset consists of many eigenvalue systems in large characteristic, and these are expected to be very hard to lift, as noted in Section \ref{sec:feasible}.
  When we restrict to primes in the range $2$ to $20$, the proportion of lifts jumps significantly.
  The systems in that range for which we could not find lifts often have very few primes satisfying the level-raising condition, large prime levels, and inert primes over the characteristic.
\end{remark}

\begin{example}\label{ex:hard-system}
  There is an ethereal eigenvalue system in characteristic $19$ over $\Q(i)$ at level \texttt{397.1} for which no $a_{\p}$ satisfies the level-raising condition for $\norm(\p) \leq 100$.
  The prime $19$ is inert in $\Q(i)$, and so adds $361$ to the norm.
  The level is also prime.
  For these reasons, we expect this eigenvalue system to be hard to lift to characteristic $0$ at higher level.
  We have not found a lift.
\end{example}

\subsection{Examples}\label{sec:examples}
In this section, we record some of the more unusual examples discovered during our computation.
In tables where Hecke eigenvalues are recorded, prime labels marked with a dagger ($\dagger$) are those dividing the characteristic of an ethereal form or the level of a lift, where we do not expect the eigenvalues to agree.
For primes dividing the level, \texttt{Magma}'s builtin algorithms cannot compute the corresponding Hecke eigenvalue, so those entries are marked with a dash ($-$).
For mod $p$ modular forms, the eigenvalue $a_{\p}$ listed for primes over the level is the scalar by which $U_{\p}$ acts.
Similarly, if $\p \mid p$, we list the eigenvalue of the endomorphism defined by Equation \ref{eq:hecke-op}, but we stress that the arithmetic meaning of this is nebulous.

\begin{example}[Lift eigenvalues in a large field]
  For $K =\Q(\sqrt{-7})$, choose a root $\omega$ of $x^{2}-x+2$, and fix the level $\mf{n} = (18\omega + 1)$ (label $(\mathtt{667.2}) = (\mathtt{23.1}) \cdot (\mathtt{29.2})$).
  Then we find an ethereal mod $53$ Bianchi eigenform $f$ with eigenvalues in $\F_{53}$.
  The characteristic zero Bianchi space of level $\Gamma_0(\mf{nr})$ where $\mf{r} = (\sqrt{-7})$ contains a large Galois orbit of eigenforms.
  This is defined over the totally real number field $E$ with minimal polynomial
  \begin{equation}
    \label{eq:min-poly}
    x^7 + x^6 - 8x^5 - 6x^4 + 17x^3 + 8x^2 - 9x + 1,
  \end{equation}
  of discriminant $281 \cdot 3227041$, currently not in the LMFDB.
  The prime $53$ splits into two factors in $E$ with residue degrees $1$ and $6$.
  These are both principal; in fact, $E$ has class number $1$.
  Fixing a root $w$ of the polynomial in Equation \ref{eq:min-poly}, the residue degree $1$ prime $\mf{p}_{53}$ is generated by $w^6 - 8w^4 + w^3 + 14w^2 - 2$.
  We observe numerically that the Hecke eigenvalues of $f$ away from $\mf{n}$ agree with those of one of the eigenforms $F$ modulo $\mf{p}_{53}$, as indicated in Table \ref{table:irrational-lift-irrational-evals-2}.

  \renewcommand{\arraystretch}{1.3}
  \begin{table}[h!]
    \begin{tabular}{lccc}
      $\p$ & $a_{\p}(f)$ & $a_{\p}(F) \Mod{ \mf{p}_{53}}$ & $a_{\p}(F)$ \\
      \midrule
      \texttt{2.1}  & $48$     & $48$     & $-w^5 + 7w^3 - 2w^2 - 10w + 4$ \\
      
      \texttt{2.2}  & $43$  & $43$  & $w$ \\
      
      \texttt{7.1}${}^{\dagger}$  & $45$  & $-$  & $-$ \\
      
      \texttt{9.1}  & $12$     & $12$     & $w^6 - 8w^4 + w^3 + 15w^2 - 2w - 4$ \\
      \texttt{11.1} & $34$     & $34$     & $w^6 + w^5 - 8w^4 - 6w^3 + 15w^2 + 8w - 3$ \\
      
      \texttt{11.2} & $8$     & $8$     & $w^5 - 7w^3 + 3w^2 + 9w - 5$ \\
      
      \texttt{23.1}${}^{\dagger}$ & $52$   & $-$   & $-$ \\
      
      \texttt{23.2} & $36$  & $36$  & $-w^6 + 7w^4 - 3w^3 - 9w^2 + 9w - 4$ \\
      \texttt{25.1} & $41$     & $41$     & $-w^6 - w^5 + 7w^4 + 5w^3 - 10w^2 - 2w - 1$ \\
      \texttt{53.1}$^{\dagger}$ & $47$     & $47$     & $3w^6 + 4w^5 - 23w^4 - 25w^3 + 45w^2 + 39w - 18$ \\
      \texttt{53.2}$^{\dagger}$ & $46$     & $46$     & $-2w^6 + w^5 + 15w^4 - 9w^3 - 22w^2 + 10w - 5$ \\
    \end{tabular}

    \caption{A liftable ethereal mod $53$ form with lift with large Hecke field.}\label{table:irrational-lift-irrational-evals-2}
\end{table}
\end{example}

\begin{example}[Large characteristic]
  We have several examples of ethereal forms for relatively large primes.
  Over $\Q(\sqrt{-7})$, at prime level $\mathtt{641.1}$ there is an ethereal form in the mod $379$ cohomology.
  Adding the ideal \texttt{2.1} to the level, we find a characteristic $0$ lift with Hecke eigenvalues in the real cubic subfield $\Q(\zeta_7)^{+} \subset \Q(\zeta_7)$, where $\zeta_7$ is a primitive $7$-th root of unity.

  The rational prime $379$ splits completely in $\Q(\zeta_7)$, and we pick the reduction map $\Z[\zeta_7] \to \F_{379}$ defined by the prime $\mf{P} = (379, \zeta_7 - 195)$ of residue degree $1$.
  The first few eigenvalues of the characteristic $0$ lift can be found in Table \ref{table:irrational-lift-irrational-evals-379}.
  \begin{table}[h!]
    \renewcommand{\arraystretch}{1.3}
    \begin{tabular}{lccc}
      $\p$ & $a_{\p}(f)$ & $a_{\p}(F) \Mod{ \mf{P}}$ & $a_{\p}(F)$ \\
      \midrule
      \texttt{2.1}$^\dagger$ & $3$ & $-$   & $-$ \\
      \texttt{2.2}  & $45$   & $45$   & $\zeta_7^{5} + \zeta_7^{4} + \zeta_7^{3} + \zeta_7^{2}$ \\
      \texttt{7.1}  & $217$  & $217$  & $\zeta_7^{5} + \zeta_7^{2} - 2$ \\
      \texttt{9.1}  & $68$   & $68$   & $-3\zeta_7^{5} - 2\zeta_7^{4} - 2\zeta_7^{3} - 3\zeta_7^{2} - 2$ \\
      \texttt{11.1} & $217$  & $217$  & $\zeta_7^{5} + \zeta_7^{2} - 2$ \\
      \texttt{11.2} & $235$  & $235$  & $3\zeta_7^{4} + 3\zeta_7^{3} - 1$ \\
      \texttt{23.1} & $11$   & $11$   & $\zeta_7^{5} - \zeta_7^{4} - \zeta_7^{3} + \zeta_7^{2} - 3$ \\
      \texttt{23.2} & $171$  & $171$  & $-\zeta_7^{4} - \zeta_7^{3} - 3$ \\			
    \end{tabular}
    \caption{Lift of ethereal mod $379$ eigenvalue system to characteristic $0$.}\label{table:irrational-lift-irrational-evals-379}
  \end{table}
\end{example}

\begin{example}[Non-Steinberg lifts for $p=2$]\label{ex:failure-2}
  Consider the ethereal mod $2$ form $f$ at level $\mf{n} = (4 + \omega)$ (label \texttt{23.1}) over $K=\Q(\sqrt{-11})$, with $\omega = \frac{1+\sqrt{-11}}{2}$.
  In the LMFDB data, we find a total of $266$ elliptic curves whose Frobenius traces match the Hecke eigenvalues of $f$ mod $2$ at primes up to norm $100$.
  The minimal level one is $E$ as in \LMRefDB{https://www.lmfdb.org/ModularForm/GL2/ImaginaryQuadratic/2.0.11.1/92.1/a/}{2.0.11.1-92.1-a}, whose conductor is $(2) \cdot \mf{n}$.
  We find another curve $E'$ at level $(26-28\omega)$ (label \texttt{2300.5}), with lift factor $\mf{r} = (2) \p_5^2$ for $\p_5 = (2-\omega)$ of norm $5$. This has LMFDB label \LMRefDB{https://www.lmfdb.org/EllipticCurve/2.0.11.1/2300.5/a/1}{2.0.11.1-2300.5-a}.
  
  Notice that $a_{\p_5}(f) = 1$, so $f$ does not satisfy the Steinberg level-raising condition at $\p_5$.
  However, as predicted by Proposition \ref{prop:ethereal-ellcurve-steinberg}, the LMFDB lists $E'$ as a quadratic twist of $E$, with additive reduction of Kodaira type $I_0^{*}$ at $\p_5$.
  
  We verify the Bianchi modular form $F$ associated to $E'$ on the level of cocycles as
  \begin{equation*}
    \overline{F} = \sum_{\mf{d} \mid \mf{r}} \alpha_{\mf{d}}(f). 
  \end{equation*}
\end{example}

\begin{example}[Failure of multiplicity $1$]\label{eg:mult-2}
  For $K = \Q(\sqrt{-3})$ and $\mf{n}$ with LMFDB label $\mathtt{931.2} = (\mathtt{7.1})^{2} \cdot (\mathtt{19.2})$, we find a non-Eisenstein mod $7$ eigensystem $f$ which appears, for Hecke operators up to norm $1000$, to have multiplicity $2$ in $H^{1}(\Gamma_0(\mf{n}), \F_7)$.
  The first few eigenvalues are listed in Table \ref{tab:mult-2}.
  
  We verify that the eigenvalue system does not appear in a level dividing $\mf{n}$, and that the intersection of the $2$-dimensional eigenspace with the image of the mod $7$ reduction map is one-dimensional.
 
\begin{table}[ht]
\centering
\begin{tabular}{l ccccccccccc}
$\p$ & \texttt{3.1} & \texttt{4.1} & \texttt{7.1}$^{\dagger}$
               & \texttt{7.2}$^{\dagger}$ & \texttt{13.1} & \texttt{13.2}
               & \texttt{19.1} & \texttt{19.2}$^{\dagger}$ & \texttt{25.1}
               & \texttt{31.1} & \texttt{31.2} \\
\midrule
$a_{\p}$ & $6$ & $5$ & $0$ & $4$ & $2$ & $1$ & $1$ & $6$ & $6$ & $0$ & $5$ \\
\end{tabular}
\caption{Hecke eigenvalues in $\mathbb{F}_7$ of the multiplicity $2$ form of level \texttt{931.2} over $\mathbb{Q}(\sqrt{-3})$.}
\label{tab:mult-2}
\end{table}
Over $\Q(\sqrt{-11})$ and levels $\mathtt{873.1} = (\mathtt{3.1})^{2}\cdot (\mathtt{97.1})$ and $\mathtt{873.2} = (\mathtt{3.1})^{2} \cdot (\mathtt{97.2})$, we also find two mod $3$ multiplicity $2$ forms which lift, after adding another factor of $\mathtt{3.1}$ to the level, to the forms \LMRefDB{https://www.lmfdb.org/ModularForm/GL2/ImaginaryQuadratic/2.0.11.1/2619.1/a/}{2.0.11.1-2619.1-a} and \LMRefDB{https://www.lmfdb.org/ModularForm/GL2/ImaginaryQuadratic/2.0.11.1/2619.2/a/}{2.0.11.1-2619.2-a}.

We also find $203$ examples of ethereal eigenvalue systems where the generalised eigenspace is strictly larger than the ordinary eigenspace, including $7$ examples with $p>3$, which suggests that the origin is not torsion in $\SL_2(\OO_K)$.
The largest dimension of a generalised eigenspace we find is $5$, occurring for $4$ distinct mod $2$ eigenvalue systems.
\end{example}

\begin{example}[A dihedral mod $7$ system with no dihedral lifts]
  A remarkably sparse ethereal mod $7$ eigenvalue system $f$ appears for $K = \Q(\sqrt{-7})$ at level $\mathtt{512.1}=(w)^{9}$ where $w$ satisfies $w^{2} -w +2 =0$, see Table \ref{tab:dihedral-example}.
  In fact, $a_{\mf{q}}(f) =0$ for any $\mf{q}$ inert in the extension $L \defeq K(\sqrt{-w})$.
  It is therefore natural to expect that the associated Galois representation $\rho_f$ is the induction of a character from $L$.
  Denoting by $\mf{P}_2$ and $\mf{P}_7$ the unique primes above $(w)$ and $(\sqrt{-7})$ respectively, we indeed find an order $24$ character $\chi$ of the ray class group $\operatorname{Cl}_{\mf{P}_2^{6} \mf{P}_7} (L) \cong \Z/24\Z \times \Z/2\Z$ valued in $\F_{49}^{\times}$, for which the Frobenius traces of  $\rho \defeq \Ind{G_K}{G_L} \chi$ agree with the Hecke eigenvalues of $f$ for primes up to norm $200$.
  The nontrivial element $\sigma$ of $\Gal(L/K)$ acts by $\chi^{\sigma} = \chi^{7}$, forcing the trace of Frobenius to land in $\F_7$.

  However, $L$ is not Galois over $\Q$ and contains an imaginary quadratic field, so it is not CM.
  By enumerating the extensions of $L$ ramified at $2$ and $7$, it is not hard to see that $\rho$ is not induced from a character of a CM extension of $K$.
  On the other hand, a simple analysis of possible infinity types shows that $\rho$ cannot lift to a characteristic $0$ theta series unless the corresponding extension is CM.

  This does not exclude the possibility that $f$ has non-dihedral lifts.
  The sparsity of Hecke eigenvalues means that level-raising primes are rare, but they do exist: up to norm $2000$ they are:
  \texttt{239.2},  \texttt{449.1},  \texttt{449.2},  \texttt{743.2},  \texttt{827.1},  \texttt{953.1}, \texttt{1163.1}, \texttt{1499.1}, \texttt{1597.2},  \texttt{1709.1},  \texttt{1723.2}, and \texttt{1933.2}.
  We note that all of these primes $\mf{q} = \mf{Q}\mf{Q}'$ split in $L$, with $\chi(\mf{Q}) = \pm 2$ and $\norm(\mf{q}) \equiv 1 \Mod 7$.

  \begin{table}[ht]
\centering
\setlength{\tabcolsep}{5pt}
\begin{tabular}{l ccccccccccc}
$\p$ & \texttt{2.2} & \texttt{9.1} & \texttt{11.1} & \texttt{11.2}
               & \texttt{23.1} & \texttt{23.2} & \texttt{25.1} & \texttt{29.1}
               & \texttt{29.2} & \texttt{37.1} & \texttt{37.2} \\
\midrule
$\mf{p}\OO_L$ & inert & inert & split & inert & split & split & inert & split & inert & inert & split \\
$a_{\p}(f)$ & $0$ & $0$ & $6$ & $0$ & $0$ & $5$ & $0$ & $3$ & $0$ & $0$ & $2$ \\
\end{tabular}\caption{Hecke eigenvalues in $\F_7$ of the ethereal system $f$ of level $\texttt{512.1} = (w)^{9}$ over $\Q(\sqrt{-7})$, with the splitting of $\p$ in $L = K(\sqrt{-w})$ in the middle row.}
\label{tab:dihedral-example}
\end{table}
\end{example}

\subsection{Results of the computations}\label{sec:results-of-comps}

In this section, we summarise various statistics relating to ethereal eigenvalue systems and their lifts. Our data consists of all ethereal, non-cyclotomic-Eisenstein eigenvalue systems over the five Euclidean imaginary quadratic fields, appearing in $H^1(\Gamma_0(\n),\F_p)$, up to $\norm(\n) \le 1000$. 

In Table \ref{table:lmfdb-lifts}, we list the proportion of ethereal eigenvalue systems in our data for which a lift was found using the LMFDB.
As the LMFDB currently only stores rational systems, this table is restricted only to the $\F_p$-valued systems.
\begin{table}[H]
  \centerline{
	\begin{tabular}{c|ccc|ccc}
		$d$  & \shortstack{Ethereal $\F_p$-systems \\ ($p < 20$)} & Admissible &\shortstack{LMFDB \\lifts} & \shortstack{Ethereal $\F_p$-systems \\ (all $p$)} & Admissible & \shortstack{LMFDB \\ lifts}  \\ \midrule
		$1$ & 152 & 136 & 122 & 188 & 136 & 122 \\
		$2$ & 390 & 301 & 265 & 654 & 304 & 268 \\ 
		$3$ & 85  & 77  & 71  & 90 & 77 & 71 \\ 
		$7$ & 289 & 230 & 200 & 500 & 233 & 202 \\ 
		$11$ & 489& 388 & 328 & 890 & 389 & 329 \\
	\end{tabular}
        }
	\caption{The second column lists the number of ethereal $\F_p$-valued eigenvalue systems. The third counts the number of forms where the existence of a lift is not precluded by the Ramanujan bound, and the fourth the number of systems with a lift in the LMFDB.}\label{table:lmfdb-lifts}
\end{table}
As noted in Remark \ref{rmk:blocking}, the Ramanujan bound predicts that not all rational systems should lift to an elliptic curve.
We call an $\F_p$-valued eigenvalue system $f$ \emph{admissible} if for every prime ideal $\p$ of $\OO_K$ coprime to $\mf{n}$ and $p$, either $\p$ is a level-raising prime for $f$, or there exists an integer $M$ with $M \equiv a_{\p}(f) \Mod p$ and $\abs{M} \le 2\sqrt{\norm(\p)}$.
 
When restricting to admissible forms, we see that the ratio of forms liftable in the LMFDB improves substantially, suggesting that the Ramanujan bound is the main obstruction to finding rational lifts.

Finally, we record the data for all ethereal eigenvalue systems valued in $\overline{\F}_p$ in Table \ref{table:all-lifts}.
\begin{table}[H]
	\begin{tabular}{c|ccc|ccc}
		$d$  & \shortstack{Ethereal \\ ($p < 20$)} & Lifts & Verified &  \shortstack{Ethereal \\ (all $p$)}& Lifts & Verified  \\ \midrule
		$1$ & $177$ & $140$ &  $87$  & $213$ & $142$ & $87$ \\ 
		$2$ & $451$ & $295$ &  $181$ & $717$ & $303$ & $184$ \\ 
		$3$ & $103$ & $83$  &  $62$  & $108$ & $84$  & $62$ \\ 
		$7$ & $363$ & $223$ &  $139$ & $574$ & $238$ & $141$ \\ 
		$11$ & $592$ & $361$ & $149$ & $996$ & $377$ & $149$ \\ 
	\end{tabular}
	\caption{Count of lifts of ethereal Bianchi modular forms, not restricted to lifts in the LMFDB.}\label{table:all-lifts}
\end{table}

\printbibliography

\end{document}